\documentclass[12pt]{amsart}
\usepackage{amsmath,amssymb,mathrsfs,enumerate,bm}
\usepackage{pstricks,pst-plot}
\usepackage[colorlinks=true,linkcolor=black,citecolor=black,urlcolor=black]{hyperref}
\usepackage{calc,tikz-cd}

\psset{unit=1pt}
\psset{arrowsize=4pt 1}
\psset{linewidth=.5pt}

\newcommand{\define}{\textbf}

\newcommand{\excise}[1]{}

\newcommand{\isom}{\cong}
\renewcommand{\setminus}{\smallsetminus}
\renewcommand{\phi}{\varphi}
\newcommand{\exterior}{\textstyle\bigwedge}
\renewcommand{\tilde}{\widetilde}

\renewcommand{\bar}{\overline}

\newcommand{\CC}{\mathbb{C}}

\newcommand{\ZZ}{\mathbb{Z}}
\newcommand{\PP}{\mathbb{P}}
\renewcommand{\AA}{\mathbb{A}}

\newcommand{\OO}{\mathcal{O}}

\newcommand{\cM}{\mathcal{M}}

\newcommand{\shfF}{\mathscr{F}}

\newcommand{\Fl}{\mathbf{Fl}}
\newcommand{\Gr}{\mathbf{Gr}}
\newcommand{\cQ}{\mathcal{Q}}

\newcommand{\pt}{\mathrm{pt}}
\newcommand{\id}{\mathrm{id}}

\newcommand{\bk}{\mathbf{k}}
\newcommand{\bp}{\mathbf{p}}
\newcommand{\bq}{\mathbf{q}}

\newcommand{\triple}{{\bm{\tau}}}
\newcommand{\eu}{\varepsilon}
\newcommand{\Eta}{\mathrm{H}}

\newcommand{\opK}{\mathrm{op}K}

\newcommand{\tor}{T\hspace{-.6ex}or}

\DeclareMathOperator{\rk}{rk}

\DeclareMathOperator{\im}{im}

\DeclareMathOperator{\Hom}{Hom}

\DeclareMathOperator{\Pf}{Pf}

\newtheorem{theorem}{Theorem}[section]
\newtheorem{lemma}[theorem]{Lemma}

\newtheorem{corollary}[theorem]{Corollary}

\newtheorem{thm}{Theorem}

\newtheorem*{lem}{Lemma}

\theoremstyle{definition}

\newtheorem{remark}[theorem]{Remark}

\newtheorem*{def*}{Definition}
\newtheorem*{ex*}{Example}

\begin{document}

\title[K-theoretic Chern class formulas for degeneracy loci]{K-theoretic Chern class formulas for vexillary degeneracy loci}
\author{David Anderson}
\email{anderson.2804@math.osu.edu}
\address{Department of Mathematics, The Ohio State University, Columbus, OH 43210}
\date{May 16, 2019}
\thanks{The author was partially supported by NSF Grant DMS-1502201.}

\begin{abstract}
Using raising operators and geometric arguments, we establish formulas for the $K$-theory classes of degeneracy loci in classical types.  
We also find new determinantal and Pfaffian expressions for classical cases considered by Giambelli: the loci where a generic matrix drops rank, and where a generic symmetric or skew-symmetric matrix drops rank.

In an appendix, we construct a $K$-theoretic Euler class for even-rank vector bundles with quadratic form, refining the Chow-theoretic class introduced by Edidin and Graham.  We also establish a relation between top Chern classes of maximal isotropic subbundles, which is used in proving the type D degeneracy locus formulas.
\end{abstract}

\maketitle

\section*{Introduction}

Given a map of vector bundles $\phi\colon E_p \to F_q$ on a variety $X$ (with subscripts indicating ranks), an old and natural problem is to determine the class of the {\em degeneracy locus} $\Omega_r \subseteq X$ where $\rk\phi\leq r$.  The prototypical setting is where $X$ is projective space and the vector bundles are sums of line bundles: then $\phi$ is represented as a $q\times p$ matrix of homogeneous forms, and one asks for the class of the locus where this matrix drops rank.  In this language, the problem was considered by nineteenth-century mathematicians.  (Cayley wrote about the case where $p=q+1$ and $r=q-1$; Salmon and Roberts treated the general case of sub-maximal rank.)  It was solved by Giambelli in 1904, who gave the answer in the form of a determinant in the homogeneous degrees of the entries of the matrix.

The modern version of this problem asks for formulas in terms of Chern classes.  Porteous's formula expresses Giambelli's determinant in the general setting of cohomology classes of degeneracy loci, and this was extended by Kempf and Laksov to loci that include all Schubert varieties in a Grassmann bundle \cite{kl}.  A more detailed history can be found in Fulton \cite{f-flags}, who also gives a further generalization of the determinantal formula to {\em vexillary} Schubert varieties in flag bundles.

Perhaps surprisingly, analogous determinantal formulas for the $K$-theory classes of degeneracy loci (i.e., their structure sheaves) were obtained only very recently.  To my knowledge, the first explicitly stated determinant appears in the remarkable paper of Hudson, Ikeda, Matsumura, and Naruse \cite{himn}.  (Some precursors can be found in work of Lascoux \cite{lascoux,l2}, Lenart \cite{lenart}, Sam \cite{sam}, and Kirillov \cite{kirillov}, as well as in the Porteous and Jacobi-Trudi formulas of Buch \cite{buch}.)

For symplectic degeneracy loci---or, locally, symmetric matrices---the story is somewhat parallel, but its culmination has been reached only very recently.  Here one has a vector bundle $V$ of even rank, equipped with a symplectic form, along with isotropic subbundles $E$ and $F$, and the problem is to express the class of the locus where $\dim(E \cap F)\geq k$.  Again, Giambelli (in 1906) was one of the first to give a general formula, expressed as a determinant.  Other determinantal formulas for cohomology classes of symplectic degeneracy loci were found in the late twentieth century, but in hindsight these were coincidences: as Pragacz discovered, the formulas are better expressed as {\em Pfaffians} (which happen to equal determinants in the special cases considered by Giambelli and others).  Pragacz also brought the formalism of Young's {\em raising operators} into this subject \cite{p1,p2}.  On the geometric side, Kazarian found the most natural extension of the Kempf-Laksov argument to the symplectic (and orthogonal) setting, and gave elegant proofs of Pfaffian formulas \cite{k}.  
This technique was profoundly refined by Buch, Kresch, and Tamvakis, who introduced and studied {\it theta-} and {\it eta-polynomials} as representatives for Schubert classes in the cohomology of isotropic Grassmannians \cite{bkt1,bkt2}.  ``Double'' versions of these polynomials were defined and shown to represent equivariant Schubert classes and degeneracy loci in \cite{wilson,im,tw,af2,t1}.

The main aim of this article is to prove determinantal and Pfaffian formulas for the $K$-theory classes of degeneracy loci.  A special case of our results directly generalizes Giambelli's formula: given a map $\phi\colon E \to F$ of vector bundles on $X$, the locus $\Omega_r$ where $\rk(\phi)\leq r$ has $K$-theory class
\begin{align}\label{e.g1}
  [\OO_{\Omega_r}] = \det\left(\sum_{m\geq0} \binom{q-r+m-1}{m}\,(-1)^m\,c_{q-r+j-i+m}(F-E) \right)_{1\leq i,j\leq p-r}.
\end{align}
\noindent
Here $p$ and $q$ are the ranks of $E$ and $F$, respectively, and $c(F-E)$ is the $K$-theoretic Chern class.  The leading term, where $m=0$, is precisely Giambelli's determinantal formula in cohomology.\footnote{Using some algebraic identities, one can deduce \eqref{e.g1} from other formulas in the literature---for example, it is equivalent to a special case of \cite[(1.2)]{himn}, but their formula uses a complete filtration of $F$ by quotient bundles.  I do not know of a reference where it is clearly stated in terms of only the virtual bundle $F-E$, although \cite[Corollaire~4.5]{lascoux} is close.}  

We will prove similar formulas for vexillary degeneracy loci, incorporating the stability properties of the cohomology formulas.\footnote{In their recent article, Hudson and Matsumura also give determinantal formulas for vexillary loci in type A \cite{hm}.  They  carry out some computations in algebraic cobordism, but so far no determinantal expressions have been found in that generality.}  Our results generalize the formulas of \cite{himn} for types A, B, and C.  In general, the loci are described by conditions of the form $\dim(E_{p_i}\cap F_{q_i})\geq k_i$, and the results of \cite{himn} are recovered as the special case where $k_i=i$ and $p_i=p$ for all $i$.

We also find formulas for the more difficult type D loci.  These require the introduction of the {\em Euler class} in $K$-theory.  This is a characteristic class for even-rank vector bundles with quadratic form which extends the Chow-theoretic Euler class defined by Edidin and Graham \cite{eg}.  The key observation is that an Euler class should be defined up to unit.  (For the Chow ring, the homogeneous units are only $\{\pm 1\}$, and indeed the Edidin-Graham Euler class is defined up to sign---but in $K$-theory there are more units.)  In contrast to the Edidin-Graham Euler class, one must invert $2$ to construct the $K$-theoretic Euler class, even if one restricts attention to bundles which are Zariski-locally trivial.  The construction of this class is described in Appendix~\ref{s.euler}, along with some fundamental properties of Chern classes of maximal isotropic subbundles.

Formulas for cohomology classes are intimately related to special polynomials in algebraic combinatorics, and often the combinatorial constructions preceded the geometric ones.  This was the case for the {\it Schubert polynomials} introduced by Lascoux and Sch\"utzenberger in 1982 as representatives for type A Schubert classes.  Analogues for other types begin with Billey and Haiman's work \cite{bh}.  Ikeda, Mihalcea, and Naruse defined double versions representing equivariant classes \cite{i,in1,imn}.  In each case, there is a Schubert polynomial for each (signed) permutation; these are defined inductively by starting at a longest element and applying difference operators.  
More detailed accounts of the Schubert polynomial story can be found in \cite{af1} and \cite{t2}.

For {\em vexillary (signed) permutations}, the Schubert polynomials have determinantal or Pfaffian formulas, which are much more efficient than the inductive definition.  
There is a simple correspondence between vexillary permutations and the rank conditions defining vexillary degeneracy loci; the interested reader can find the combinatorics worked out in \cite{bl,af1,af3,jordan}.  
Since present article is focused on degeneracy loci, we mainly use the language of rank conditions rather than permutations.

As in cohomology, the $K$-theoretic degeneracy locus formulas are related to  combinatorial polynomials.  In type A, these are the {\em Grothendieck polynomials} of Lascoux and Sch\"utzenberger.  Grothendieck polynomials in other classical types have been introduced recently by Ikeda, Kirillov, and Naruse \cite{in,kn}.  As with Schubert polynomials, the Grothendieck polynomials are defined inductively for each (signed) permutation; by replacing variables with Chern classes, these polynomials can be interpreted as formulas for quite general degeneracy loci.  Conversely, by replacing Chern classes with appropriately specialized variables, the results presented here lead to determinantal and Pfaffian---or more generally, raising-operator---formulas for vexillary Grothendieck polynomials.

Since degeneracy loci are ubiquitous in algebraic geometry, applications of these determinantal and Pfaffian formulas should be abundant.  For instance, variants of the formula \eqref{e.g1} and those of \cite{himn} have been used to compute the holomorphic Euler characteristics of Brill-Noether loci in Jacobians of curves \cite{act}.  In fact, Kempf and Laksov state that their formula was motivated by applications to the study of exactly such loci, and their paper was immediately followed by a computation of the cohomology class of a Brill-Noether locus \cite{kll}.

Both the proofs and the formulas themselves simplify considerably by employing the methods of \cite{af2}.  The vexillary formulas are not merely a generalization from the Grassmannian case---including vexillary loci streamlines the argument, even for Grassmannians.  Furthermore, our formulas are written in terms of $K$-theoretic Chern classes, rather than the useful but less familiar Segre classes of \cite{himn}.  (See Remark~\ref{r.himn} for the translation between the two.)    
Perhaps more significantly, the same method leads directly to both the discovery and the proof of the type D formulas.

The essential idea of the proof is to reformulate the techniques of Kempf-Laksov and Kazarian in terms of raising operators, using only basic properties of Chern classes and pushforwards via projective bundles.  The extra flexibility gained by including the vexillary loci (instead of only Grassmannian loci) allows one to separate the geometry from the algebra.  There are four main steps.  First, one identifies a product formula for the class of a simple kind of vexillary degeneracy locus, called ``dominant''---these loci are never Grassmannian, except in trivial cases.  Second, by using the algebra of raising operators, one transforms this product into a determinant (in type A) or Pfaffian (in other types).  Third, one constructs a resolution of any vexillary degeneracy locus by a dominant one, via a tower of projective bundles; a special case of the desired formula is deduced by pushing forward the determinantal or Pfaffian formula for the dominant locus.  A final algebraic manipulation produces the general formula, using some universal relations among Chern classes in types B, C, and D.

The determinantal formula for type A is Theorem~\ref{t.mainA}, and the Pfaffian formulas for type C, B, and D are Theorems~\ref{t.mainC}, \ref{t.mainB}, and \ref{t.mainD}, respectively.  Theorems~\ref{t.thetaC}, \ref{t.thetaB}, and \ref{t.etaD} provide ``$\beta$-theta'' and ``$\beta$-eta polynomial'' formulas for loci involving conditions on coisotropic bundles; these are notationally more cumbersome than the Pfaffian formulas, but conceptually very similar.  Extensions of Giambelli's formula for symmetric and skew-symmetric matrices, parallel to \eqref{e.g1}, are stated in Corollaries~\ref{c.pf} and \ref{c.pfD}.

In a sense, the $K$-theory formulas proved here are the most general ones of their kind: a theorem of Bressler and Evens shows that Schubert classes are sensitive to the choice of desingularization in all cohomology theories beyond $K$-theory \cite{be}.  Several authors have made significant progress in understanding generalized cohomology theories of flag varieties and the related degeneracy loci (see \cite{hm1,hm2,cpz,czz} and references therein).  Formulas in this context must necessarily be of a different flavor, however.

\smallskip
\noindent
{\it Acknowledgements.}  I especially thank William Fulton for the ongoing collaboration that has led to my understanding of the determinantal formulas presented here.  I also thank Linda Chen and Nicola Tarasca for suggestions for improving the $K$-theory formulas in type A, and Anders Buch, Dan Edidin, William Graham, Eric Marberg, Brendan Pawlowski, and Richard Thomas for comments on the manuscript.  Finally, I must acknowledge my debt to the authors of \cite{himn} for their groundbreaking work on the subject, and I also thank them for remarks on earlier drafts.

Much of this work was done during the thematic program on Combinatorial Algebraic Geometry at the Fields Institute, and while visiting the University of Oslo.

\smallskip
\noindent
{\it Conventions.}  Throughout, $X$ is an irreducible variety over an algebraically closed field.  In the body of the paper, the results are expressed in {\it connective $K$-theory}, $CK^*(X)$, which is a graded algebra over $\ZZ[\beta]$ interpolating $K^\circ(X)$ and the Chow ring $A^*(X)$.  A digest of the relevant facts, along with a short development of the foundations, can be found in Appendix~\ref{s.kthy}.

I am deliberately avoiding the issue of interpreting degeneracy locus formulas.  To a first approximation, the standard constructions of intersection theory give meaning to the class $[\Omega]$; more details are given in \cite{himn} and \cite{act}.  The reader may assume $X$ is smooth to simplify matters, and see also Remark~\ref{r.interpret}.  (An expert reader may notice that I have abused notation for degeneracy locus classes if $X$ is singular: more precisely, the Chern class formulas should be transported to homology by capping with the fundamental class $[X]$.)

\section{Type A degeneracy loci}\label{s.typeA}

Given two flags vector bundles on $X$,
\[
  E_{p_1} \subseteq \cdots \subseteq E_{p_s} \subseteq V
\]
and
\[
  F_{q_1} \subseteq \cdots \subseteq F_{q_s} \subseteq V,
\]
with $\rk(E_{p_i}) = p_i$ and $\rk(V/F_{q_i})=q_i$, we consider degeneracy loci corresponding to \define{triples} $\triple = (\bk,\bp,\bq)$, defined by
\begin{align*}
 \Omega_\triple &= \{ x\in X\,|\, \dim(E_{p_i} \cap F_{q_i}) \geq k_i \text{ for } 1\leq i\leq s\} \\
   &=\{ x\in X\,|\, \rk(E_{p_i} \to \tilde{F}_{q_i})\leq r_i \text{ for } 1\leq i\leq s\},
\end{align*}
where $\tilde{F}_{q_i} = V/F_{q_i}$ and $r_i=p_i-k_i$.

In order for the rank conditions to be feasible and nontrivial, we must require some inequalities on the triple $\triple$.  From the setup, it is clear that $0<k_1<\cdots <k_s$, $0<p_1\leq \cdots \leq p_s$, and $q_1\geq \cdots \geq q_s>0$.  Furthermore, the sequence
\[
  \lambda_{k_i} = q_i-p_i+k_i
\]
must be weakly decreasing---so it can be extended to a partition $\lambda=\lambda(\triple)$ by setting $\lambda_k=\lambda_{k_i}$ whenever $k_{i-1}<k\leq k_i$. 
For example, if $\triple = (\bk,\bp,\bq)=(\; 1\;2\;4\, ,\; 1\;3\;6\, , \; 4\;4\;4\, \; )$, then $\lambda = (4,3,2,2)$.

Such degeneracy loci are called \define{vexillary}, because they correspond to Schubert varieties for vexillary permutations, i.e., permutations avoiding the pattern $2\;1\;4\;3$.  A recipe for writing down the vexillary permutation $w(\triple)$ associated to a triple can be found in \cite{af1}.

An equivalent point of view is as follows: start with maps of bundles,
\[
E_{p_1} \hookrightarrow E_{p_2} \hookrightarrow \cdots \hookrightarrow E_{p_s} \xrightarrow{\phi} \tilde{F}_{q_1} \twoheadrightarrow \tilde{F}_{q_2} \twoheadrightarrow \cdots \twoheadrightarrow \tilde{F}_{q_s},
\]
of ranks indicated by the subscripts, and define the locus $\Omega_\triple$ according to the second description above.  (The previous situation is recovered by setting $V=E_{p_s} \oplus \tilde{F}_{q_1}$ and considering the graph of $\phi$.)

Schubert loci in Grassmann bundles are a notable special case.  Here we start with a vector bundle $V$ on a variety $Y$, and set $X=\Gr(p,V)$.  Taking all $p_i=p$, and the single bundle $E=E_p\subset V$ to be the tautological subbundle on $X$, the locus $\Omega_\triple$ is exactly the Schubert variety $\Omega_\lambda$ as defined in \cite{himn}.

More generally, when $X=\Fl(V)$ is a (partial or complete) flag bundle, with the $E_{p_i}$ forming the tautological flag, the loci $\Omega_\triple$ coincide with Schubert loci $\Omega_{w(\triple)}$ for an associated {\em vexillary permutation}, as explained in \cite{af1}.  These are the universal cases for degeneracy locus problems, and as such they play a role in interpreting the formulas we give.

\begin{thm}\label{t.mainA}
In $CK^*(X)$, we have
\begin{align*}
  [\Omega_\triple] &= \Delta_{\lambda(\triple)}( c(1),\ldots,c(k_s); \beta ) \\
  &:=\det\left( \sum_{m\geq0} \binom{\lambda_i+m-1}{m}\beta^m\, c_{\lambda_i-i+j+m}(i) \right)_{1\leq i,j\leq k_s} ,
\end{align*}
where $c(k_i) = c(V-E_{p_i}-F_{q_i}) = c(\tilde{F}_{q_i}-E_{p_i})$ for $1\leq i\leq s$, and $c(k)=c(k_i)$ when $k_{i-1}<k\leq k_i$.
\end{thm}
\noindent
Specializing $\beta=0$ immediately recovers the classical determinantal formulas in cohomology, and setting $\beta=-1$ yields the formula \eqref{e.g1} of the introduction, for the triple $\triple=(p-r,p,q)$.

The point of view we emphasize here is that {\em raising operators} simplify the formulas.  Let $T_i$ be the operator which raises the index of $c(i)$ by one, and write $R_{ij} = T_i/T_j$.  The $(i,j)$ entry of the determinant in the theorem can be written compactly as
\[
  (1-\beta T_i)^{-\lambda_i}\cdot c_{\lambda_i-i+j}(i) ,
\]
using the generalized binomial coefficient identity $\binom{-\lambda}{m} = (-1)^m\binom{\lambda+m-1}{m}$.  Using the Vandermonde identity, the determinant itself can be expressed as
\begin{align}\label{e.raisingA}
  \left( \prod_{1\leq i<j\leq k_s} (1-R_{ij}) \right) \cdot \left( \prod_{i=1}^{k_s} (1-\beta T_i)^{-\lambda_i} c_{\lambda_i}(i) \right).
\end{align}
(We leave this simple exercise in matrix algebra to the reader.)  In fact, it is this formulation which we will prove.

\begin{proof}
The proof follows the outline of \cite{af2}.

\medskip
\noindent
{\bf Basic case.}  First we examine the simplest locus, where $s=1$, $p_1=1$, and $k_1=1$.  (This means $\lambda_1=q_1$.)  So $E_1$ is a line bundle, and the degeneracy locus $\Omega_\triple$ is defined by $E_1 \subseteq F_{q_1}$, or equivalently, $E_1 \to \tilde{F}_{q_1}$ is zero.  By basic properties of Chern classes (Property (a) in Appendix~\ref{s.kthy}), we have
\begin{align*}
  [\Omega_\triple] &= c_{q_1}(\tilde{F}_{q_1}\otimes E_1^*) \\
           &= (1-\beta T)^{-q_1} c_{q_1}(\tilde{F}_{q_1}-E_1),
\end{align*}
as desired.

\medskip
\noindent
{\bf Dominant case.}  Next we establish a product formula for the case where $p_i=k_i=i$ for all $i$.  (This means $\lambda_i=q_i$ for all $i$.)  The locus is defined by requiring $E_i\subseteq F_{q_i}$ for all $i$, or equivalently, requiring $E_i/E_{i-1} \to \tilde{F}_{q_i}$ to be zero.  So by repeated application of the basic case, we see
\begin{align*}
  [\Omega_\triple] &= \prod_{i=1}^s (1-\beta T_i)^{-q_i} c_{q_i}(\tilde{F}_{q_i}-E_i/E_{i-1}) \\
  &= \prod_{i=1}^s (1-\beta T_i)^{-q_i} \prod_{i=1}^s c_{q_i}(\tilde{F}_{q_i}-E_i/E_{i-1}) \\
  &=  \left( \prod_{i=1}^s (1-\beta T_i)^{-q_i} \right) \left( \prod_{1\leq i<j\leq s}(1-R_{ij})\right) \cdot c_{q_1}(1) \cdots c_{q_s}(s),
\end{align*}
which is the required formula.

(The equality
\begin{align}\label{e.prod-raise}
\prod_{i=1}^s c_{q_i}(\tilde{F}_{q_i}-E_i/E_{i-1}) = \left( \prod_{1\leq i<j\leq s}(1-R_{ij})\right) \cdot c_{q_1}(1) \cdots c_{q_s}(s),
\end{align}
used between the second and third lines, is formal algebra.  Setting $t_i=-c_1(E_i/E_{i-1})$, the Whitney formula lets us write
\[
  \prod_{i=1}^s c_{q_i}(\tilde{F}_{q_i}-E_i/E_{i-1}) = \prod_{j=1}^s\left[ c(j)\cdot \prod_{i=1}^{j-1} (1-t_i) \right]_{q_i};
\]
now argue by induction on $s$ to obtain \eqref{e.prod-raise}, using Property (b) from Appendix~\ref{s.kthy}.  This is the same argument as in \cite{af2}.)

\medskip
\noindent
{\bf Main case.}  Now we push forward the above expression to obtain the formula for the case where $k_i=i$.  (So $\lambda_i=q_i-p_i+i$.)  This is done by resolving the locus $\Omega_\triple$, via exactly the same Kempf-Laksov-type resolution described in \cite{af1}, \cite{af2} and \cite{himn}.  

Specifically, we form a sequence of projective bundles
\[
  X=X_0 \xleftarrow{\pi^{(1)}} X_1 = \PP(E_{p_1}) \xleftarrow{\pi^{(2)}} X_2 = \PP(E_{p_2}/D_1) \leftarrow \cdots \xleftarrow{\pi^{(s)}} X_s = \PP(E_{p_s}/D_{s-1}),
\]
where $D_i/D_{i-1} \subseteq E_{p_i}/D_{i-1}$ is the tautological line bundle on $X_i$.  Let $\pi\colon X_s \to X$ be the composition of all these maps.

On $X_s$, there is the locus $\tilde\Omega$ where $D_i/D_{i-1} \to \tilde{F}_{q_i}$ is zero for all $i$.  This is an instance of the dominant case, with $\tilde{\lambda}_i=q_i$ for all $i$, and $\tilde{c}(i) = c(F_{q_i}-D_i)$, so we have
\begin{align}\label{e.res-form}
  [\tilde\Omega] = \left( \prod_{i=1}^s (1-\beta T_i)^{-q_i} \right) \left( \prod_{1\leq i<j\leq s}(1-R_{ij})\right) \cdot \tilde{c}_{q_1}(1) \cdots \tilde{c}_{q_s}(s)
\end{align}
in $CK^*(X_s)$.

We need to compute $\pi_*[\tilde\Omega]$.  To do this, first use the Whitney formula to write
\[
  \tilde{c}(i) = c(\tilde{F}_{q_i}-D_i) = c(\tilde{F}_{q_i}-E_{p_i})\cdot c(E_{p_i}-D_i) =  c(i)\cdot c(E_{p_i}/D_i).
\]
Next, observe that $E_{p_i}/D_i$ is the tautological quotient bundle for the projective bundle $\pi^{(i)}\colon X_i \to X_{i-1}$, and $c(i)$ is pulled back from $X$.

We claim that
\begin{align}\label{e.push}
  \pi^{(i)}_*\left[ (1-\beta T_i)^{-q_i}\cdot \tilde{c}_{q_i}(i) \right] &= (1-\beta T_i)^{-q_i+p_i-i} \cdot c_{q_i-p_i+i}(i) \\
  &= (1-\beta T_i)^{-\lambda_i} \cdot c_{\lambda_i}(i).  \nonumber
\end{align}
To see this, let us write 
\begin{align*}
(1-\beta T_i)^{-q_i}\cdot \tilde{c}_{q_i}(i) &= \left[(1-\beta)^{-q_i}\cdot \tilde{c}(i)\right]_{q_i} \\
&= \left[ (1-\beta)^{-q_i}\cdot c(i)\cdot c(E_{p_i}/D_i) \right]_{q_i}.
\end{align*}
The pushforward $\pi^{(i)}_*$ lowers degree by $p_i-i$, so using the projection formula together with Property (c) of Appendix~\ref{s.kthy}, we have
\begin{align*}
  \pi^{(i)}_*\left[ (1-\beta )^{-q_i} c(i)\cdot c(E_{p_i}/D_i) \right]_{q_i} &= \left[ (1-\beta )^{-q_i} \cdot c(i) \cdot  \pi^{(i)}_*c(E_{p_i}/D_i) \right]_{q_i-p_i+i} \\
  &= \left[ (1-\beta )^{-q_i} \cdot c(i) \cdot  (1-\beta)^{p_i-i} \right]_{q_i-p_i+i} \\
  &= (1-\beta T_i)^{-q_i+p_i-i}\cdot c_{q_i-p_i+i}(i),
\end{align*}
as claimed.

Applying this calculation to \eqref{e.res-form} (for each $i$ from $s$ to $1$), we find
\begin{align*}
[\Omega]= \pi_*[\tilde\Omega] &= \left( \prod_{i=1}^s (1-\beta T_i)^{-\lambda_i} \right) \left( \prod_{1\leq i<j\leq s}(1-R_{ij})\right) \cdot {c}_{\lambda_1}(1) \cdots {c}_{\lambda_s}(s),
\end{align*}
as required.

\medskip
\noindent
{\bf General case.}  Given a general triple $\triple=(\bk,\bp,\bq)$, one can ``inflate'' it by adding entries to obtain a triple $\triple'=(\bk',\bp',\bq')$ having $k'_i=i$, without essentially changing the corresponding degeneracy locus.  This is done just as in \cite[\S1.4]{af2} (cf.~\cite[(3.6)]{mac2}).  Suppose $k_i-k_{i-1}>1$, and note that $q_{i-1}>q_i$ or $p_{i-1}<p_i$; we consider the latter case, the former being similar.  Insert $(k_{i}-1, p_i-1, q_i)$ between $(k_{i-1},p_{i-1},q_{i-1})$ and $(k_i,p_i,q_i)$, to obtain a new triple $\triple'$ having $\lambda(\triple')=\lambda(\triple)$; note that $\lambda_{k_i-1}=\lambda_{k_i}$.  The condition $\dim(E_{p_i}\cap F_{q_i})\geq k_i$ (from $\triple$) implies the only new condition in $\triple'$, namely $\dim(E_{p_i-1}\cap F_{q_i})\geq k_i-1$, so the degeneracy locus is the same.

The entries of the determinant do change when passing from $\triple$ to $\triple'$, but the determinant itself does not, by the algebraic identity of \cite[App.~A.3, Lemma 1]{af2}.  Let us explain this lemma in the case described above, so $\triple'$ comes from $\triple$ by inserting $(k_i-1,p_i-1,q_i)$.  The difference in the matrices appears in rows $(k_{i-1}+1)$ through $(k_i-1)$, where the classes $c(k_i)=c(V-E_{p_i}-F_{q_i})$ are replaced by $c(k_i-1)=c(V-E_{p_i-1}-F_{q_i})=c(k_i)\cdot c(E_{p_i}/E_{p_i-1})$.  So 
\[
  c_m(k_i-1) = c_m(k_i)+a c_{m-1}(k_i)
\]
for all $m$, where $a=c_1(E_{p_i}/E_{p_i}-1)$.

Writing $\ell=\lambda_{k_i-1}=\lambda_{k_i}$, and focusing on the rows $k_i-1$ and $k_i$, the determinant for $\triple'$ appears as
\begin{align*}
& \left|\begin{array}{cccc}\cdots & c_\ell(k_i-1) & c_{\ell+1}(k_i-1) & \cdots \\ \cdots & c_{\ell-1}(k_i) & c_\ell(k_i) & \cdots\end{array}\right| \\
& \qquad = \left|\begin{array}{cccc}\cdots & c_\ell(k_i) + a c_{\ell-1}(k_i) & c_{\ell+1}(k_i) + a c_{\ell}(k_i) & \cdots \\ \cdots & c_{\ell-1}(k_i) & c_\ell(k_i) & \cdots\end{array}\right| \\
& \qquad = \left|\begin{array}{cccc}\cdots & c_\ell(k_i)  & c_{\ell+1}(k_i)  & \cdots \\ \cdots & c_{\ell-1}(k_i) & c_\ell(k_i) & \cdots\end{array}\right| +  \left|\begin{array}{cccc}\cdots &  a c_{\ell-1}(k_i) &  a c_{\ell}(k_i) & \cdots \\ \cdots & c_{\ell-1}(k_i) & c_\ell(k_i) & \cdots\end{array}\right| \\
& \qquad = \left|\begin{array}{cccc}\cdots & c_\ell(k_i)  & c_{\ell+1}(k_i)  & \cdots \\ \cdots & c_{\ell-1}(k_i) & c_\ell(k_i) & \cdots\end{array}\right| + 0,
\end{align*}
with both rows multiplied by $(1-\beta T)^{-\ell}$, so this equals the determinant for $\triple$.
\end{proof}

\begin{remark}\label{r.himn}
In \cite{himn}, the determinantal formula is given in terms of {\em Segre classes} in $CK$-theory.  Since these are somewhat less standard than Chern classes, we have preferred to avoid them in the main formulas.  They are defined in \cite{himn} by setting
\[
  s(E^*-F^*;u) = \frac{1}{1+u^{-1}\beta}c(F-E;u+\beta),
\]
where $s(V;u)=\sum s_m(V) u^m$ and $c(V;u) = \sum c_m(V) u^m$ are the Segre and Chern polynomials, respectively.  To simplify the notation a bit, let us write $c'_m(F|E) = s_m(E^*-F^*)$.  (Warning: These Segre classes can be nonzero in negative degree.  And while these classes are well-defined for virtual bundles $E^*-F^*$, they are {\em not} compatible with the Whitney sum formula.)  In this notation, \cite[Theorem~3.10]{himn} says
\begin{align}\label{e.himn}
 [\Omega_\triple] = \det\left(\sum_{m\geq 0} \binom{i-j}{m}\beta^m\, c'_{\lambda_i-i+j+m}(\tilde{F}_{q_i}|E) \right)_{1\leq i,j\leq p},
\end{align}
for $E=E_p$, with $\triple$ having $k_i=i$ for $1\leq i\leq p$, and all $p_i=p$.

In fact, our formula is equivalent to the natural generalization of \eqref{e.himn} to vexillary loci,
\begin{align}\label{e.himn2}
 [\Omega_\triple] = \det\left(\sum_{m\geq 0} \binom{i-j}{m}\beta^m\, c'_{\lambda_i-i+j+m}(\tilde{F}_{q_i}|E_{p_i}) \right)_{1\leq i,j\leq k_s},
\end{align}
for any triple $\triple$.

To see this is equal to the formula of Theorem~\ref{t.mainA}, first observe
\begin{align*}
  c'_m(F|E) &= \sum_{i=0}^\infty \binom{m-1+i}{i}\beta^i\, c_{m+i}(F-E) \\
     &= \sum_{i=0}^{\infty} \binom{-m}{i}(-\beta)^i c_{m+i}(F-E) \\
     &= (1-\beta T)^{-m} \cdot c_m(F-E),
\end{align*}
where $T$ is the operator raising the index of $c(F-E)$ by one.  So the $(i,j)$ entry of the matrix in \eqref{e.himn2} is
\begin{align*}
 & \sum_{m\geq 0} \binom{i-j}{m}\beta^m\,c'_{\lambda_i-i+j+m}(\tilde{F}_{q_i}|E_{p_i})\\
 &\qquad = \sum_{m\geq 0} \binom{i-j}{m}\beta^m\, (1-\beta T_i)^{-\lambda_i+i-j-m} \cdot c_{\lambda_i-i+j+m}(\tilde{F}_{q_i}-E_{p_i}) \\
 & \qquad=  \sum_{m\geq 0} \binom{i-j}{m}\beta^m\, \sum_{l\geq 0} \binom{-\lambda_i+i-j-m}{l}(-\beta)^l \cdot c_{\lambda_i-i+j+m+l}(\tilde{F}_{q_i}-E_{p_i}) \\
 & \qquad= \sum_{t\geq 0} \sum_{m+l=t} \binom{i-j}{m}\binom{\lambda_i-i+j+m_l-1}{l}\beta^t\, c_{\lambda_i-i+j+t}(\tilde{F}_{q_i}-E_{p_i}) \\
 & \qquad= \sum_{t\geq 0} \binom{\lambda_i+t-1}{t}\beta^t\, c_{\lambda_i-i+j+t}(\tilde{F}_{q_i}-E_{p_i}) ,
\end{align*}
which is the $(i,j)$ entry of Theorem~\ref{t.mainA}.  (The last line uses the Cauchy identity
\[
  \sum_{a+b=c}\binom{p}{a}\binom{q}{b} = \binom{p+q}{c},
\]
valid for generalized binomial coefficients.)
\end{remark}

\begin{remark}\label{r.eagon-northcott}
The problem of computing the $K$-theory class of $\Omega_r$ is naturally related to that of constructing complexes which resolve ideals generated by minors of a matrix.  As noted in \cite[\S4]{l2}, in the case of maximal minors, this is the well-known Eagon-Northcott complex.  Such a complex determines a formula in $K$-theory; however, packaging the resulting formula as determinant seems to involve nontrivial algebra.
\end{remark}

\begin{remark}\label{r.interpret}
As in cohomology, the degeneracy locus formulas require either some genericity hypothesis, or some interpretation.  We briefly describe the situation for our purposes, which is a variation on the standard construction in intersection theory \cite[\S14]{f-it}; a more detailed treatment is in \cite{himn} and \cite[\S1]{act}.  The universal case for these formulas is that of a flag bundle $\Fl(V) \to X$, with the tautological flag $S_\bullet$ playing the role of $E_\bullet$, and with $F_\bullet$ being a flag of bundles pulled back from $X$.  The locus $\Omega_\triple$ is identified with a Schubert variety $\Omega_{w(\triple)}\subseteq \Fl(V)$, for a {\it vexillary permutation} $w(\triple)$.  The fact that Schubert varieties have rational singularities means that the resolution $\pi\colon\tilde\Omega \to \Omega$ constructed in the ``main case'' of the proof does in fact satisfy $\pi_*[\tilde\Omega] = [\Omega]$.

In general, the meaning of $[\Omega]$, for a degeneracy locus $\Omega\subseteq X$, is the pullback of the corresponding locus in $\Fl(V)$ via the section $X \to \Fl(V)$ determined by the given flag $E_\bullet$.

The same comments apply to the symplectic degeneracy loci of the following sections, with the type C, B, or D flag bundles in place of $\Fl(V)$.
\end{remark}

\section{Type C degeneracy loci}\label{s.typeC}

Here $V$ is a vector bundle of rank $2n$, equipped with a symplectic form, as well as two flags of isotropic subbundles
\[
  V \supset E_{p_1} \supseteq E_{p_2} \supseteq \cdots \supseteq E_{p_s}
\]
and
\[
  V \supset F_{q_1} \supseteq F_{q_2} \supseteq \cdots \supseteq F_{q_s},
\]
with $E_{p_i}$ and $F_{q_i}$ having rank $n+1-p_i$ and $n+1-q_i$, respectively.  The degeneracy locus associated to a (type C) triple $\triple = (\bk,\bp,\bq)$ is defined as in type A,
\[
  \Omega_\triple = \{ x\in X \,|\, \dim(E_{p_i}\cap F_{q_i})\geq k_i \text{ for } 1\leq i\leq s\},
\]
but the inequalities required on $\triple$ are different: the setup demands $0<k_1<\cdots<k_s$, $p_1\geq \cdots \geq p_s>0$, and $q_1\geq \cdots \geq q_s>0$, and we also require that 
\[
  k_i-k_{i-1} \leq (p_{i-1}-p_i) + (q_{i-1}-q_i)
\]
for $2\leq i\leq s$.  This means the sequence
\[
 \lambda_{k_i} = p_i+q_i-1
\]
is strictly decreasing, and can be extended to a {\em strict} partition $\lambda = \lambda(\triple) = (\lambda_1 > \lambda_2 > \cdots > \lambda_{k_s} >0)$, by setting $\lambda_k = \lambda_{k_i} + (k_i-k)$ for $k_{i-1}<k\leq k_i$.

\begin{thm}\label{t.mainC}
In $CK^*(X)$, we have
\begin{align*}
  [\Omega_\triple] &= \Pf_{\lambda(\triple)}( c(1),\ldots,c(k_s);\beta ) \\
  & := \Pf( M ),
\end{align*}
where $c(k_i) = c(V-E_{p_i}-F_{q_i})$ for $1\leq i\leq s$, $c(k)=c(k_i)$ if $k_{i-1}<k\leq k_i$, and $M = (m_{ij})$ is the skew-symmetric matrix with entries
\begin{align}
  m_{ij} &= \left( \frac{1-R_{ij}}{1+R_{ij}-\beta T_i} \right)\cdot \left( (1-\beta T_i)^{k_s-i-\lambda_i} \cdot c_{\lambda_i}(i) \right) \left( (1-\beta T_j)^{k_s-j-\lambda_j}\cdot c_{\lambda_j}(j) \right) \label{e.mij}
\end{align}
for $1\leq i<j\leq k_s$, with $m_{ji} = -m_{ij}$ and $m_{ii}=0$.  (If $k_s$ is odd, augment the matrix by an initial $0^{\mathrm{th}}$ row and column, with entries $m_{0j}=(1-\beta T_j)^{k_s-j-\lambda_j}\cdot c_{\lambda_j}(j)$.)
\end{thm}
\noindent
Specializing $\beta=0$ recovers the cohomological Pfaffian formulas of \cite{k,af1}.

The matrix entries can be written out explicitly by expanding the raising operator.  It is not hard to check the identity
\begin{align*}
  \frac{1-R_{ij}}{1+R_{ij}-\beta T_i} 
  &= 1+ \sum_{m> 0} \beta^m\,T_i^m \\ &\qquad + \sum_{\ell>0}(-1)^\ell \sum_{m\geq 0}\left( \binom{\ell+m-1}{m} + \binom{\ell+m}{m} \right)(\beta T_i)^m R_{ij}^\ell.
\end{align*}
So writing
\begin{align} \label{e.mij2}
  m'_{ij} &=  c_{\lambda_i}(i)\,c_{\lambda_j}(j) + \sum_{m>0} \beta^m c_{\lambda_i+m}(i)\,c_{\lambda_j}(j)  \\ 
  &  \quad + \sum_{\ell>0} \sum_{m\geq 0}  (-1)^\ell \left( \binom{\ell+m-1}{m} + \binom{\ell+m}{m} \right) \beta^{m}\, c_{\lambda_i+\ell+m}(i)\, c_{\lambda_j-\ell}(j) , \nonumber
\end{align}
we have $m_{ij} = \left( (1-\beta T_i)^{-\lambda_i+1} \,(1-\beta T_j)^{-\lambda_j}\right)\cdot m'_{ij}$.  Note that the matrix can also be written as $M=S\cdot M'\cdot S^{\mathit{t}}$, where $M'=(m'_{ij})$ and $S$ is a diagonal matrix with $i^{\mathrm{th}}$  entry $(1-\beta T_{i})^{k_s-i-\lambda_{i}}$.

Using Knuth's identity \cite[(4.3)]{knuth},
\[
  \prod_{i<j} \frac{T_i^{-1}-T_j^{-1}}{T_i^{-1}+T_j^{-1}-\beta} = \Pf\left( \frac{T_i^{-1}-T_j^{-1}}{T_i^{-1}+T_j^{-1}-\beta}\right),
\]
together with the general identity $\Pf(S\cdot M'\cdot S^{\mathit{t}}) = \det(S)\cdot \Pf(M')$, the Pfaffian in the theorem can be written more compactly in terms of raising operators:
\begin{align}\label{e.raisingC}
 [\Omega_\triple] &= R^{(k_s)}\cdot \left(\prod_{i=1}^{k_s} (1-\beta T_i)^{-\lambda_i} c_{\lambda_i}(i) \right),
\end{align}
where
\begin{align*}
R^{(k_s)} = \left( \prod_{1\leq i<j\leq k_s}  \frac{1-R_{ij}}{1+R_{ij}-\beta T_i}\right) \cdot \left(\prod_{i=1}^{k_s} (1-\beta T_i)^{k_s-i} \right).
\end{align*}
As in type A, it is the formulation of \eqref{e.raisingC} that we will prove.

\begin{proof}
Again, the proof follows the same structure as in \cite{af2}.  The main difference between type C and type A appears in the ``dominant'' case, where the Vandermonde product is replaced by the Schur Pfaffian product.

\medskip
\noindent
{\bf Basic case.}  This is the same as in type A, with only notational changes.  We consider $s=1$, $p_1=n$, and $k_1=1$, so $\lambda_1=n+q_1-1$, $E_n$ is a line bundle, and the degeneracy locus $\Omega_\triple$ is defined by $E_n \subseteq F_{q_1}$.  Noting that the rank of $V/F_{q_1}$ is equal to $\lambda_1=n+q_1-1$, we have
\begin{align*}
  [\Omega_\triple] &= c_{n+q_1-1}(V/F_{q_1}\otimes E_n^*) \\
           &= (1-\beta T)^{-\lambda_1} c_{\lambda_1}(V-F_{q_1}-E_n).
\end{align*}

\medskip
\noindent
{\bf Dominant case.}  Next is the case where $k_i=i$ for all $i$, and $p_i=n+1-i$, so $\lambda_i=n-i+q_i$.  Let us write $D_i = E_{n+1-i}$, so $D_i$ has rank $i$.  The locus can be defined by requiring $D_i \subseteq F_{q_i}$ for all $i$, or equivalently, $D_i/D_{i-1} \subseteq F_{q_i}/D_{i-1}$.  Using the fact that the bundles are isotropic, this is the same as requiring $D_i/D_{i-1} \to D_{i-1}^\perp/F_{q_i}$ to be zero for all $i$.   Noting that $\lambda_i=\rk(D_{i-1}^\perp/F_{q_i})$, and applying the basic case, we find a product of factors of the form $(1-\beta T_i)^{-\lambda_i} c_{\lambda_i}(D_{i-1}^\perp/D_{i-1} - F_{q_i}/D_{i-1} - D_i/D_{i-1})$.  Here we use the symplectic form again to identify $D_{i-1}^\perp = (V/D_{i-1})^*$, and also $V^* = V$, so
\[
  c(D_{i-1}^\perp/D_{i-1} - F_{q_i}/D_{i-1} - D_i/D_{i-1}) = c(V-D_i-F_{q_i})\,c(D_{i-1}-D_{i-1}^*).
\]
Writing $t_i=-c_1(D_i/D_{i-1})$, one has
\[
  c(D_{i-1}^*) = \prod_{j=1}^{i-1} \frac{1+t_j-\beta t_j}{1-\beta t_j},
\]
so
\[
 c(D_{i-1}-D_{i-1}^*) = \prod_{j=1}^{i-1} \frac{(1-t_j)(1-\beta t_j)}{1+t_j-\beta t_j}.
\]
Putting these together, and using Property (b),
\begin{align*}
  [\Omega_\triple] &= \prod_{i=1}^s (1-\beta T_i)^{-\lambda_i} c_{\lambda_i}(D_{i-1}^\perp/D_{i-1} - F_{q_i}/D_{i-1} - D_i/D_{i-1}) \\
  &= \prod_{i=1}^s (1-\beta T_i)^{-\lambda_i} \prod_{i=1}^s \left[c(V-D_i-F_{q_i})\, \prod_{j=1}^{i-1} \frac{(1-t_j)(1-\beta t_j)}{1+t_j-\beta t_j} \right]_{\lambda_i} \\
  &=  \left( \prod_{i=1}^s (1-\beta T_i)^{-\lambda_i+s-i} \right) \left( \prod_{1\leq i<j\leq s}\frac{1-R_{ij}}{1+R_{ij}-\beta T_i} \right) \cdot c_{\lambda_1}(1) \cdots c_{\lambda_s}(s) \\
   &=  R^{(s)}\cdot \left( \prod_{i=1}^s (1-\beta T_i)^{-\lambda_i} c_{\lambda_i}(i) \right),
\end{align*}
as claimed.

\medskip
\noindent
{\bf Main case.}  The argument for this case, where $k_i=i$ for all $i$, is similar to the corresponding case in type A.  Following Kazarian, we resolve the locus $\Omega_\triple$ via a tower of projective bundles,
\[
  X=X_0 \xleftarrow{\pi^{(1)}} X_1 = \PP(E_{p_1}) \xleftarrow{\pi^{(2)}} X_2 = \PP(E_{p_2}/D_1) \leftarrow \cdots \xleftarrow{\pi^{(s)}} X_s = \PP(E_{p_s}/D_{s-1}),
\]
with $D_i/D_{i-1} \subseteq E_{p_i}/D_{i-1}$ being the tautological line bundle on $X_i$, and $\pi\colon X_s\to X$ the composition.

Consider the locus $\tilde\Omega \subseteq X_s$ where $D_i \subseteq F_{q_i}$ for all $i$, an instance of the dominant case, with $\pi\colon \tilde\Omega \to \Omega$ birational.  Letting $\tilde\lambda_i = n-i+q_i$ and $\tilde{c}(i) = c(V - D_{i} - F_{q_i})$, we have
\begin{align}
  [\tilde\Omega] =R^{(s)}\cdot \left( \prod_{i=1}^s (1-\beta T_i)^{-\tilde\lambda_i} c_{\tilde\lambda_i}(i) \right)
\end{align}
by the dominant case.  Just as in type A (see \eqref{e.push}), the application of Property (c) to the pushforward $\pi^{(i)}$ yields
\[
  \pi^{(i)}_*[ (1-\beta T_i)^{-\tilde\lambda_i+s-i} \cdot \tilde{c}_{\tilde\lambda_i}(i)] = (1-\beta T_i)^{-\lambda_i+s-i} \cdot c_{\lambda_i}(i).
\]
Putting these together, we find the desired formula:
\[
  \pi_*[\tilde\Omega] = R^{(s)}\cdot \left( \prod_{i=1}^s (1-\beta T_i)^{-\lambda_i} c_{\lambda_i}(i) \right).
\]

\medskip
\noindent
{\bf General case.}  As in type A, we inflate $\triple=(\bk,\bp,\bq)$ to a triple $\triple'=(\bk',\bp',\bq')$ with $k'_i=i$, to reduce to the main case.  Suppose, for some $\ell$, that $k_\ell-k_{\ell-1}>1$ and $p_{\ell-1}>p_\ell$.  (The case $q_{\ell-1}>q_\ell$ is similar.)  Inserting $(k_\ell-1,\,p_\ell+1,\,q_\ell)$ between $(k_{\ell-1},\,p_{\ell-1},\,q_{\ell-1})$ and $(k_\ell,\,p_\ell,\,q_\ell)$ produces a triple $\triple'$ with $\lambda(\triple')=\lambda(\triple)$, and note in particular that
\[
  \lambda_{k_\ell-1}=\lambda_{k_\ell}+1.
\]
Since $\dim(E_{p_\ell}\cap F_{q_\ell})\geq k_\ell$ implies $\dim(E_{p_\ell+1}\cap F_{q_\ell}) \geq k_\ell-1$, the degeneracy locus is unchanged.

The Pfaffian formula is also invariant under the substitution of $\triple'$ for $\triple$, although the entries of the matrix are different.  This identity requires a bit more algebra than before.  Let
\[
  c'(k_{\ell}-1) = c(V-E_{p_\ell+1}-F_{q_\ell}) = c(k_\ell-1)\cdot c(E_{p_\ell}/E_{p_\ell+1}),
\]
and $c'(k)=c(k)$ for all $k\neq k_\ell-1$.  (Repeating this argument to change each $k$ from $k_{i-1}+1$ through $k_i-1$ produces the Chern classes appearing in the formula associated to $\triple'$.)  With $a=c_1(E_{p_\ell}/E_{p_\ell+1})$, we have
\begin{align}\label{e.expand}
  c'_m(k_{\ell}-1) = c_m(k_{\ell}-1) + a c_{m-1}(k_{\ell}-1).
\end{align}
We wish to show that
\begin{align}\label{e.general}
  & R^{(k_s)}\cdot \left( \prod_{i=1}^{k_s} (1-\beta T_i)^{-\lambda_i} c_{\lambda_i}(i)\right)  \\
  & \quad =  R^{(k_s)}\cdot \left( \prod_{i=1}^{k_s} (1-\beta T_i)^{-\lambda_i} c'_{\lambda_i}(i)\right), \nonumber
\end{align}
assuming inductively that $[\Omega]$ is given by the RHS.  
In view of \eqref{e.expand}, the RHS of \eqref{e.general} is equal to the LHS, plus
\[
  a\cdot  \prod (1-\beta T_i)^{-\lambda_i} \cdot R^{(k_s)} \cdot (\cdots c_{\lambda_{k_{\ell}-1}-1}(k_{\ell}-1) \,c_{\lambda_{k_\ell}}(k_{\ell})\cdots  ).
\]
Since $\lambda_{k_{\ell}-1}-1=\lambda_{k_\ell}$ and $c(k_{\ell}-1)=c(k_\ell)$, this term is the Pfaffian of a matrix whose $(k_\ell-1)^{\mathrm{st}}$ and $k_\ell^{\mathrm{th}}$ rows (and columns) are identical, and it therefore vanishes.  (The only nontrivial part of this observation is that the $(k_\ell-1,k_\ell)$ entry is zero; this is the content of Lemma~\ref{l.relation}, below.)
\end{proof}

In contrast with the type A situation, there are universal relations among Chern classes for bundles equipped with a symplectic form.  The relations needed to prove the general case of the theorem are stated in the following lemma, whose proof uses special instances of the dominant and main cases of the theorem.

\begin{lemma}\label{l.relation}
Let $c(1)=c(2)=c(V-E_p-F_q)$, and $\lambda=p+q-1$.  Then 
\[
  (1-\beta T_1)^{-m} (1-\beta T_2)^{-m} \frac{1-R_{12}}{1+R_{12}-\beta T_1} c_{\lambda}(1) \, c_{\lambda}(2) = 0
\]
for any integer $m$.
\end{lemma}

\begin{proof}
First we treat the case $m=\lambda$, using the geometric arguments from the ``main'' and ``dominant'' cases of the previous proof.

Let $\Omega$ be the locus defined by $\dim(E_p\cap F_q)\geq 2$.  Without changing the locus, include the redundant condition $\dim(E_{p+1}\cap F_q) \geq 1$.  Set the following notation:
\begin{align*}
\lambda'_1 &= (p+1)+q-1 = \lambda+1,  & c'(1) &= c(V-E_{p+1}-F_q) = c(1)\cdot(1+a), \\
\lambda'_2 &= p+q-1 = \lambda,  & c'(2) &= c(V-E_{p}-F_q) = c(2),
\end{align*}
where $a=c_1(E_p/E_{p+1})$, and also write $R = \frac{1-R_{12}}{1+R_{12}-\beta T_1}$.  By the main case of the previous proof, we have
\[
  [\Omega] = (1-\beta T_1)^{-\lambda'_1+1} (1-\beta T_2)^{-\lambda'_2}\cdot R \cdot c'_{\lambda'_1}(1)\,c'_{\lambda'_2}(2),
\]
which can be rewritten as
\begin{align}\label{e.r1}
[\Omega] = (1-\beta T_1)^{-\lambda} (1-\beta T_2)^{-\lambda}\cdot R \cdot (T_1+a) c_{\lambda}(1)\,c_{\lambda}(2).
\end{align}

On the other hand, applying the main case construction to the ``improper triple'' $(1\;2,\; p\;p,\; q\;q)$, we form a diagram
\[
\begin{tikzcd}
 {} &  X_1 = \PP(E_p) \ar{dl}[swap]{\pi^{(1)}}  \\
X  &  \Gr(2,E_p) \ar{l}{\phi^{(1)}} &   X_2 =\PP(E_p/D_1) = \Fl(1,2;E_p) \ar{l}{\phi^{(2)}} \ar[start anchor={[xshift=-9ex]} ]{ul}[swap]{\pi^{(2)}}  \\
\Omega \ar[hook]{u} & \bar\Omega \ar{l} \ar[hook]{u} & \makebox[\widthof{$X_2 =\PP(E_p/D_1) = \Fl(1,2;E_p)$}][l]{$\tilde{\Omega}$} \ar{l} \ar[hook, start anchor={[xshift=-14ex]} , end anchor={[xshift=-14ex]} ]{u}
\end{tikzcd}
\]
where $\tilde\Omega = \{ D_1\subseteq F_q,\; D_2 \subseteq F_q\}$ (as in the previous proof), and $\bar\Omega = \{D_2\subseteq F_q\}$.  Note that $\phi^{(1)}$ maps $\bar\Omega$ birationally to $\Omega$, so $\phi^{(1)}_*[\bar\Omega] = [\Omega]$.  But $\phi^{(2)}$ is a $\PP^1$ bundle, and so is its restriction to $\tilde\Omega$, since $\tilde{\Omega} = (\phi^{(2)})^{-1}\bar\Omega$.  Therefore $\phi^{(2)}_*[\tilde\Omega] = -\beta[\bar\Omega]$, by property (c) of Appendix~\ref{s.kthy}.  It follows that $\pi_*[\tilde\Omega] = -\beta[\Omega]$, where $\pi\colon X_2\to X$ is the composition.

Applying the same algebra as in the main and dominant cases to compute $\pi_*$, we find
\begin{align}\label{e.r2}
 -\beta[\Omega] &= (1-\beta T_1)^{-\lambda+1} (1-\beta T_2)^{-\lambda} \cdot R \cdot c_{\lambda}(1)\,c_{\lambda}(2).
\end{align}
Adding $\beta$ times the RHS of \eqref{e.r1} to the RHS of \eqref{e.r2} yields the relation
\begin{align*}
 0 &= (1-\beta T_1)^{-\lambda}(1-\beta T_2)^{-\lambda} R \cdot c_{\lambda}(1)\, c_{\lambda}(2) \\
 & \qquad - \beta T_1 (1-\beta T_1)^{-\lambda}(1-\beta T_2)^{-\lambda} R \cdot c_{\lambda_1}(1)\, c_{\lambda_2}(2) \\
 & \qquad + \beta \,a\, (1-\beta T_1)^{-\lambda}(1-\beta T_2)^{-\lambda} R \cdot c_{\lambda}(1)\, c_{\lambda}(2) \\
& \qquad + \beta T_1 (1-\beta T_1)^{-\lambda}(1-\beta T_2)^{-\lambda} R \cdot c_{\lambda}(1)\, c_{\lambda}(2) \\
&= (1+\beta\,a)(1-\beta T_1)^{-\lambda}(1-\beta T_2)^{-\lambda} R\cdot c_{\lambda}(1)\,c_{\lambda}(2).
\end{align*}
Since $a=c_1(E_p/E_{p+1})$ is nilpotent, the factor $(1+\beta\,a)$ is a unit, and the claim for $m=\lambda$ follows.

The relation for general $m$ comes from the formal algebraic identity
\begin{align*}
(1-\beta T_1)^{-m} (1-\beta T_1)^{-m} R \cdot c_k(1) c_k(2) = R \cdot c_k(1) c_k(2),
\end{align*}
valid for any series $c=c(1)=c(2)=1+c_1+c_2+\cdots$, for any $k\geq 0$, and any $m\in \ZZ$.  To prove it, subtract the right side from the left and expand the raising operator $( (1-\beta T_1)^{-m} (1-\beta T_1)^{-m} - 1 )R$ as
\begin{align*}
 &  \left( \left(\sum_{a\geq 0} (-\beta)^a \binom{-m}{a} (T_1+T_2-\beta T_1 T_2)^a \right) - 1 \right) \frac{1-T_1/T_2}{1+T_1/T_2-\beta T_1}\\
 &\quad=   \left(\sum_{a\geq 1} (-\beta)^a \binom{-m}{a} (T_1+T_2-\beta T_1 T_2)^a  \right) \frac{T_2-T_1}{T_2+T_1-\beta T_1 T_2} \\
 &\quad=  \left(\sum_{a\geq 1} (-\beta)^a \binom{-m}{a} (T_1+T_2-\beta T_1 T_2)^{a-1}  \right) (T_2-T_1).
\end{align*}
This is an anti-symmetric series in positive powers of $T_1$ and $T_2$, so when applied to $c_k(1) c_k(2)$ (for $c(1)=c(2)$), the result is zero.
\end{proof}

As a special case of Theorem~\ref{t.mainC}, we obtain a formula extending Giambelli's formula for symmetric matrices.   
Consider a symmetric morphism $\phi\colon W\to W^*$, i.e., one so that the dual $\phi^\vee \colon W^{**} = W \to W^*$ is equal to $\phi$, and suppose $W$ has rank $n$.  Let $c=c(W^*-W)$, and write
\[
  c'_m = (1-\beta T)^{-1}\cdot c_m = \sum_{k\geq 0} \beta^k c_{m+k}
\]
for any $m\in \ZZ$.

\begin{corollary}\label{c.pf}
The locus $\Omega_r$ where $\rk(\phi)\leq r$ has class
\begin{align*}
  [{\Omega_r}] &= \Pf_{(n-r,\ldots,1)}(c,\ldots,c;\beta) \\
        &= \Pf\left(m_{ij} \right)_{1\leq i,j\leq n-r},
\end{align*}
where
\begin{align*} 
  m_{ij} &=  c'_{\lambda_i}\,c'_{\lambda_j} + \sum_{m>0} \beta^m\, c'_{\lambda_i+m}\,c'_{\lambda_j}  \\ 
  &  \qquad + \sum_{\ell>0} \sum_{m\geq 0}  (-1)^\ell \left( \binom{\ell+m-1}{m} + \binom{\ell+m}{m} \right) \beta^m \, c'_{\lambda_i+\ell+m}\, c'_{\lambda_j-\ell} , \nonumber
\end{align*}
and $\lambda_i = n-r+1-i$ for $1\leq i\leq n-r$.
\end{corollary}

\noindent
Setting $\beta=-1$ produces a formula for $[\OO_{\Omega_r}]$ in $K^\circ(X)$.

\begin{proof}
Let $V=W\oplus W^*$, equipped with the canonical symplectic form (as in \cite[\S6.2]{fp}): $\langle v\oplus f, \, w\oplus g \rangle = f(w) - g(v)$.  The condition that the morphism $\phi$ be symmetric is equivalent to requiring that its graph $E = \Gamma_\phi \subseteq V$ be isotropic with respect to this form.  Let $F=W\oplus 0 \subseteq V$, another isotropic subbundle.  Then $\rk(\phi)\leq r$ is equivalent to $\dim(E\cap F) \geq n-r$, so the locus $\Omega_r$ is the degeneracy locus associated to $\triple =(n-r,1,1)$.

The formula of the corollary now follows from that of the theorem, with a little explanation.  The Chern classes are $c(V-E-F) = c(V/E - F) = c(W^*-W)$, since $E$ and $F$ are both isomorphic to $W$.
\end{proof}


\section{Type B degeneracy loci}\label{s.typeB}

A similar argument establishes a Pfaffian formula for a odd-orthogonal degeneracy loci.  Here $V$ has rank $2n+1$ and is equipped with a nondegenerate quadratic form, and the isotropic subbundles $E_{p_i}$ and $F_{q_i}$ have rank $n+1-p_i$ and $n+1-q_i$, respectively.  The triple $\triple=(\bk,\bp,\bq)$ and associated strict partition $\lambda(\triple)$ are the same as in type C, as is the definition of the degeneracy locus $\Omega_\triple$.

We will assume there exists a maximal isotropic subbundle $F\supseteq F_{q_1}$, and let $M=F^\perp/F$.  This line bundle is independent of the choice of $F$, since $M\isom \det(V)$, and furthermore, the identification $V\isom V^*$ means that $M^2$ is trivial.

In order to have Chern class formulas for orthogonal degeneracy loci, one needs to invert $2$, so from now on we work with $CK^*(X)[\frac{1}{2}]$.  This makes $c_1(M)=0$, since
\[
  0 = c_1(M^2) = c_1(M)\cdot (2+\beta c_1(M)),
\]
and $(2+\beta c_1(M))$ is now a unit.  However, we will continue to include $M$ in the formulas, as an indication of the underlying geometry.

Let $c(k_i) = c(V-E_{p_i}-F_{q_i}-M)$ for $1\leq i\leq s$, and $c(k)=c(k_i)$ for $k_{i-1}<k\leq k_i$.

\begin{thm}\label{t.mainB}
In $CK^*(X)[\frac{1}{2}]$, we have
\begin{align*}
  [\Omega_\triple] &= \Pf^B_{\lambda(\triple)}( c(1),\ldots,c(k_s);\beta ),
\end{align*}
where $\Pf^B_{\lambda(\triple)}( c(1),\ldots,c(k_s);\beta )$ is the Pfaffian of the skew symmetric matrix $M=(m_{ij})$ where the entries $m_{ij}$ are
\begin{align}
  \left( \frac{1-R_{ij}}{1+R_{ij}-\beta T_i} \right)\cdot \left( \frac{(1-\beta T_i)^{k_s-i-\lambda_i+1}}{2-\beta T_i} \cdot c_{\lambda_i}(i) \right) \left( \frac{(1-\beta T_j)^{k_s-j-\lambda_j+1}}{2-\beta T_j}\cdot c_{\lambda_j}(j) \right), \label{e.mijB}
\end{align}
augmented by a $0^{\mathrm th}$ row and column in case $k_s$ is odd, with entries $m_{0j} = (1-\beta T_j)^{k_s-j-\lambda_j+1}(2-\beta T_j)^{-1}\cdot c_{\lambda_j}(j)$.
\end{thm}

As before, this is equivalent to a raising operator formula:
\begin{align}\label{e.raisingB}
 [\Omega_\triple] &= R^{(k_s)}\cdot \left(\prod_{i=1}^{k_s}\frac{1-\beta T_i}{2-\beta T_i} (1-\beta T_i)^{-\lambda_i} c_{\lambda_i}(i) \right).
\end{align}

The only difference between this and the type C situation arises in the basic case.  Here we need a computation.

\begin{lemma}\label{l.basicB}
Let $\Omega = \Omega_{(1,n,q)}$ be the locus where $E_n=D_1 \subseteq F_q$.  Then
\[
 [\Omega] = \frac{1-\beta T}{2-\beta T}(1-\beta T)^{n+q-1}\cdot c_{n+q-1}(V-D_1-F_q-M),
\]
where $T^k\cdot c_m(V-D_1-F_q-M) = c_{m+k}(V-D_1-F_q-M)$.
\end{lemma}

\noindent
This is equivalent to \cite[Corollary~7.4]{himn}, or its integral form \cite[Lemma~7.3]{himn}.

The proof of Theorem~\ref{t.mainB} now proceeds exactly as for Theorem~\ref{t.mainC}, using Lemma~\ref{l.basicB} to treat the basic case.

\section{Type D degeneracy loci}\label{s.typeD}

As usual, the case where $V$ has rank $2n$ and is equipped with a nondegenerate quadratic form presents the most complications.  We will use the notation of \cite{af2}, so the isotropic bundles $E_{p_i}$ and $F_{q_i}$ have ranks $n-p_i$ and $n-q_i$, respectively.  A {\em type D triple} is $\triple = (\bk,\bp,\bq)$ with $0<k_1<\cdots <k_s$, $p_1\geq \cdots \geq p_s\geq 0$, and $q_1\geq \cdots \geq q_s\geq0$, satisfying inequalities which can be expressed simply as follows: if $\triple^+$ is formed by adding $1$ to each $p_i$ and each $q_i$, then $\triple^+$ is a type C triple.

The degeneracy locus is defined similarly, by
\[
\Omega_\triple = \{ x\in X \,|\, \dim(E_{p_i}\cap F_{q_i})\geq k_i \},
\]
but one should take care when interpreting the scheme structure; this should be interpreted as the closure of the locus where equality holds (see \cite[\S6]{fp} or \cite[\S6.3.2]{t2}).  We also have a strict partition $\lambda=\lambda(\triple)$ defined by
\[
  \lambda_{k_i} = p_i+q_i,
\]
with other parts filled in minimally.  In contrast to the previous cases, it may happen that $\lambda_{k_s}=0$, and this is included in the data.  In this case, the definition of $\Omega_\triple$ includes the condition that $\dim(E_0\cap F_0)\geq k_s$; being the closure of the locus where equality holds means that $\dim(E_0\cap F_0)\equiv k_s \pmod 2$ on all of $\Omega_\triple$.

The formulas require a bit more setup, both algebraically and geometrically.  On the algebraic side, we will use elements $c_\alpha(i)$, $e_\alpha(i)$, and $d_\alpha(i)=c_\alpha(i)+\sigma(i)e_\alpha(i)$, where $\sigma(i)\in \{1,0,-1\}$, and operators $\delta_i$ will act by sending $\sigma(i)$ to $0$.  This action extends to monomials $d_{\alpha_1}(1)\cdots d_{\alpha_s}(s)$, with the effect of replacing $d_{\alpha_i}(i)$ by $c_{\alpha_i}(i)$.  Further details about this algebra can be found in \cite[Appendix A]{af2} (with slightly different notation).

The geometry involves {\em Euler classes} $e(E_p,F_q)$ for isotropic subbundles $E_p$ and $F_q$.  Choosing maximal isotropic bundles $E\supseteq E_p$ and $F\supseteq F_q$, these can be written as
\begin{align*}
  e_m(E_p,F_q) = \begin{cases} (-1)^{\dim(E\cap F)} \gamma(E,F)\, c_{p+q}(E/E_p+F/F_q) & \text{ if }m=p+q; \\ 0 & \text{ if } m\neq p+q, \end{cases}
\end{align*}
where $\gamma(E,F)\in CK^0(X)$ is the canonical square root of $c(V-E-F;\beta)$ defined in Appendix~\ref{s.euler}.  Results from the appendix also show that $e(E_p,F_q)$ is independent of the choice of $E$ and $F$.

For $1\leq i\leq s$, let $c(k_i) = c(V-E_{p_i}-F_{q_i})$ and $e(k_i) = e(E_{p_i},F_{q_i})$; for $k_{i-1}<k\leq k_i$, set $c(k)=c(k_i)$ and $e(k)=e(k_i)$.  We will write $d(i) = c(i) + (-1)^i e(i)$.  

\begin{thm}\label{t.mainD}
In $CK^*(X)[\frac{1}{2}]$, we have
\begin{align*}
  [\Omega_\triple] &= \Pf^D_{\lambda(\triple)}( d(1),\ldots,d(k_s);\beta ) \\
  &:= \Pf(M),
\end{align*}
where $M=(m_{ij})$ is the skew-symmetric matrix with entries
\begin{align}
  m_{ij} &=  \frac{1-\delta_i\delta_j R_{ij}}{1+\delta_i\delta_j(R_{ij}-\beta T_i)} \cdot \frac{(1-\beta \tilde T_i)^{k_s-i-\lambda_i+1}}{2-\beta \tilde T_i} \cdot \frac{(1-\beta \tilde T_j)^{k_s-j-\lambda_j+1}}{2-\beta \tilde T_j} \nonumber \\
  & \qquad \cdot \big(  c_{\lambda_i}(i)-(-1)^{k_s} e_{\lambda_i}(i) \big) \big( c_{\lambda_j}(j)+(-1)^{k_s} e_{\lambda_j}(j) \big) ,\label{e.mijD}
\end{align}
where $\tilde T_i = \delta_i T_i$.  (When $k_s$ is odd, the matrix is augmented by a $0^{\mathrm th}$ row and column, setting $m_{0j} = (1-\beta\tilde T_j)^{k_s-j-\lambda_j+1}(2-\beta\tilde T_j)^{-1}\cdot ( c_{\lambda_j}(j) + e_{\lambda_j}(j) )$.)
\end{thm}

\noindent
As before, specializing $\beta=0$ recovers the cohomology formulas from \cite{k,af2}.

Let us unpack the definition of $m_{ij}$ by writing 
\begin{align} \label{e.mij2D}
  m'_{ij} &=  \big( c_{\lambda_i}(i) - (-1)^{k_s} e_{\lambda_i}(i) \big)\,\big( c_{\lambda_j}(j) + (-1)^{k_s} e_{\lambda_j}(j) \big)+ \sum_{m>0} \beta^m c_{\lambda_i+m}(i)\,c_{\lambda_j}(j)  \\ 
  &  \quad + \sum_{\ell>0} \sum_{m\geq 0}  (-1)^\ell \left( \binom{\ell+m-1}{m} + \binom{\ell+m}{m} \right) \beta^{m}\, c_{\lambda_i+\ell+m}(i)\, c_{\lambda_j-\ell}(j) . \nonumber
\end{align}
Then
\[
  m_{ij} = \frac{(1-\beta\tilde T_i)^{k_s-i-\lambda_i+1}}{2-\beta\tilde T_i}\frac{(1-\beta\tilde T_j)^{k_s-j-\lambda_j+1}}{2-\beta\tilde T_j}\cdot m'_{ij}.
\]
As in type C, we can write $M = S\cdot M' \cdot S^{\mathit{t}}$, where $M'=(m'_{ij})$ and $S$ is diagonal with $i^{\mathrm{th}}$ entry $(1-\beta \tilde T_i)^{k_s-i-\lambda_i+1}(2-\beta\tilde T_i)^{-1}$.

As usual, we will prove an equivalent raising operator formula:
\begin{align}\label{e.raisingD}
 [\Omega_\triple] &=\tilde R^{(k_s)}\cdot \left(\prod_{i=1}^{k_s} \frac{1-\beta\tilde T_i}{2-\beta\tilde T_i} (1-\beta \tilde T_i)^{-\lambda_i} d_{\lambda_i}(i) \right),
\end{align}
where
\begin{align*}
 \tilde R^{(k_s)} = \left( \prod_{1\leq i<j\leq k_s}  \frac{1-\delta_i\delta_j R_{ij}}{1+\delta_i \delta_j (R_{ij}-\beta T_i)}\right) \left(\prod_{i=1}^{k_s} (1-\beta\tilde T_i)^{k_s-i} \right),
\end{align*}
To see that this is equivalent to Theorem~\ref{t.mainD}, one uses results of \cite[Appendix A]{af2} to show that the raising operator $\tilde R^{(k_s)}$ can be expressed as a Pfaffian.  For this, it suffices to check that the operators
\[
  H_{xy} = \frac{T_y - \delta_x \delta_y T_x}{T_y + \delta_x \delta_y T_x( 1-\beta T_y)}
\] 
satisfy the identity
\[
  H_{xy} H_{xz} H_{yz} = H_{yz} - (2\delta_x-1)(2\delta_y-1)H_{xz} + (2\delta_x-1)(2\delta_y-1)H_{xy}
\]
modulo relations $\delta_i^2=\delta_i$, which is easily done using computer algebra.

The proof of Theorem~\ref{t.mainD} follows the same pattern as in other types.  We will need a lemma for the basic case.

\begin{lemma}\label{l.basicD}
Let $\Omega=\Omega_{(1,n-1,q)}$ be the locus where $E_{n-1}=D_1\subseteq F_q$.  Then
\[
  [\Omega] = \frac{1-\beta T}{2-\beta T}(1-\beta T)^{-n+1-q}\cdot c_{n-1+q}(V-D_1-F_q) - \frac{1}{2} e(D_1,F_q),
\]
where
\[
  e(D_1,F_q) = (-1)^{\dim(E\cap F)} \gamma(E,F) c_{n-1+q}(E/D_1+F/F_q)
\]
for maximal isotropic bundles $E\supseteq D_1$ and $F\supseteq F_q$.
\end{lemma}

\begin{proof}
Consider the quadric bundle $\pi\colon \cQ(V) \to X$, and let $S\subseteq V$ be the tautological line bundle.  As in \cite[Appendix B]{af2}, we have
\[
  [\Omega] = \pi_*\left( c_{n-1}(E/D_1\otimes S^*) c_q(F/F_q\otimes S^*) ef \right),
\]
where $e=[\PP(E)]$ and $f=[\PP(F)]$ in $CK^*(\cQ(V))$.  Letting $h=c_1(S^*)$ and $c=c(V-E-F)$, we will use the formulas
\begin{align}
  \pi_*\left( \frac{h^k}{(1+\beta h)^{k-a}} ef \right) &= \frac{1-\beta T}{2-\beta T}(1-\beta T)^{-a} c_k \label{basicD1}
\end{align}
for $k>0$ and $a\geq 0$, and
\begin{align}
  \pi_*(ef) &= 1-\frac{1}{2}(1+(-1)^{\dim(E\cap F)}) \gamma(E,F) \nonumber \\ 
    &= \frac{1-\beta T}{2-\beta T}\cdot c_0 -\frac{1}{2}(-1)^{\dim(E\cap F)}\gamma(E,F), \label{basicD2}
\end{align}
both formulas taking place in $CK^*(X)[\frac{1}{2}]$.

The general case of \eqref{basicD1} is easily deduced from the case $a=k$.  To establish this case, let $\iota\colon \cQ(V) \hookrightarrow \PP(V)$ be the inclusion.  We have two expressions for $\iota^*\iota_*(f)$:
\begin{align*}
  \iota^*\iota_*(f) &= c_1({S^*}^2)\cdot f = h (2+\beta h) f
\intertext{and}
  \iota^*\iota_*(f) &= \iota^*[\PP(F)] = c_n(V/F\otimes S^*) = (1-\beta T)^{-n} c_n(V-S-F).
\end{align*}
Using $2+\beta h = (2+\beta t)(1+\beta t)^{-1}$, where $t=c_1(S)$, and noting that $t$ acts as $-T$ on $c(V-S-F)$, we obtain
\[
  h\cdot f = \frac{1-\beta T}{2-\beta T} (1-\beta T)^{-n} c_n(V-S-F).
\]
Multiplying this by $h^{k-1} = (-t)^k/(1+\beta t)^k$ gives
\[
  h^k\cdot f = \frac{1-\beta T}{2-\beta T} (1-\beta T)^{-n+1-k} c_{n-1+k}(V-S-F).
\]
The formula for $\pi_*(h^k e f)$ follows, by factoring the projection $\rho\colon \PP(E) \to X$ into the inclusion $\PP(E) \hookrightarrow \cQ(V)$ and the projection $\pi\colon \cQ(V) \to X$.

Formula \eqref{basicD2} follows from Theorem~\ref{t.push} of Appendix~\ref{s.euler}, using the identity $\gamma(E,F)=c(V-E-F;\beta/2)$ (Equation~\eqref{e.half} of Appendix~\ref{s.euler}) and writing
\[
  \frac{1-\beta T}{2-\beta T} = 1 - \frac{1}{2}\left(\frac{1}{1-(\beta/2)T}\right).
\]

To finish proving the lemma, we compute:
\begin{align*}
  [\Omega] &= \pi_*\left( c_{n-1}(E/D_1\otimes S^*) c_q(F/F_q\otimes S^*) ef \right)\\
   &= \pi_*\left(  \sum_{k=0}^{n-1+q} c_{n-1+q-k}(E/D_1+F/F_q)\left(\frac{h^k}{(1+\beta h)^{k-n+1-q}}\right)  ef \right) \\
  &= \sum_{k=0}^{n-1+q} c_{n-1+q-k}(E/D_1+F/F_q)\, \pi_*\left( \left(\frac{h^k}{(1+\beta h)^{k-n+1-q}}\right) ef \right) \\
  &= c_{n-1+q}(E/D_1+F/F_q)\,\pi_*( (1+\beta h)^{n-1+q} ef ) \\
  &\qquad + \sum_{k=1}^{n-1+q} c_{n-1+q-k}(E/D_1+F/F_q)\,\left(\frac{1-\beta T}{2-\beta T}\right) (1-\beta T)^{-n+1-q} c_k.
\end{align*}
The claim is then proved by expanding the factor
\begin{align*}
  \pi_*( (1+\beta h)^{n-1+q} ef ) &= \pi_*(ef) + \sum_{k=1}^{n-1+q} \binom{n-1+q}{k}\beta^k \frac{1-\beta T}{2-\beta T}(1-\beta T)^{-k} c_k \\
  &= \sum_{k=0}^{n-1+q} \binom{n-1+q}{k}\beta^k \frac{1-\beta T}{2-\beta T}(1-\beta T)^{-k} c_k \\
  &\qquad - \frac{1}{2}(-1)^{\dim(E\cap F)}\gamma(E,F) \\
  &= \frac{1-\beta T}{2-\beta T}(1-\beta T)^{-n+1-q} c_0 - \frac{1}{2}(-1)^{\dim(E\cap F)}\gamma(E,F)
\end{align*}
and recombining terms.
\end{proof}

Now we prove the theorem, focusing on the points where the argument differs from other types.

\begin{proof}[Proof of Theorem~\ref{t.mainD}]
The basic case is precisely Lemma~\ref{l.basicD}, so we start with the dominant case, where the main new features arise.   
Here $k_i=i$ for all $i$, and $p_i=n-i$, so $\lambda_i=n-i+q_i$.  Let $D_i=E_{n-i}/E_{n-1-i}$.  As in type C, we obtain a product of factors coming from iterations of the basic case: $[\Omega] = \zeta_1\cdots \zeta_s$, where the $i^{\mathrm{th}}$ factor is
\begin{align}
\zeta_i &= \frac{1-\beta T}{2-\beta T}(1-\beta T)^{-n+i-q_i} c_{n-i+q_i}(D_{i-1}^\perp/D_{i-1} - D_i/D_{i-1} - F_{q_i}/D_{i-1}) \\
& \quad - \frac{1}{2}(-1)^{\dim(E/D_{i-1}\cap F/D_{i-1})}\gamma(E/D_{i-1},F/D_{i-1}) \,c_{n-i+q_i}(E/D_i+F/F_{q_i}).  \nonumber
\end{align}

First, we claim that for $a\geq 1$,
\[
  t_i^a \cdot \zeta_i =  \frac{1-\beta T}{2-\beta T}(1-\beta T)^{-n+i-q_i}c_{n-i+q_i+a}(D_{i-1}^\perp/D_{i-1} - D_i/D_{i-1} - F_{q_i}/D_{i-1}),
\]
where $t_i = -c_1(D_i/D_{i-1})$.  This will show that $t_i$ acts as $\delta_i T_i$ on the factor $\zeta_i$.

It suffices to treat the case $i=1$, since one can replace $V$ by $D_{i-1}^\perp/D_{i-1}$.  We may also assume $a=1$, since for $a>1$ we have $t_1^{a-1}\cdot c_{n+q_1}(V - D_1 - F_{q_1}) = c_{n+q_1+a-1}(V - D_1 - F_{q_1})$, by property (b) of Appendix~\ref{s.kthy}.  Now we compute:
\begin{align*}
  t_1\cdot \zeta_1 &= \frac{1-\beta T}{2-\beta T}(1-\beta T)^{-n+1-q_1}\sum_{k=1}^{n+q_1} t_1^k c_{n+q_1-k}(V/F_{q_1})  \\
& \quad + \frac{1}{2}(-1)^{\dim(E\cap F)}\gamma(E,F) \,c_{n+q_1}(E+F/F_{q_1}) \\
&= \frac{1-\beta T}{2-\beta T}(1-\beta T)^{-n+1-q_1}c_{n+q_1}(V-D_1-F_{q_1}) - \frac{1}{2}c_{n+q_1}(V/F_{q_1})  \\
& \quad + \frac{1}{2}(-1)^{\dim(E\cap F)}\gamma(E,F) \,c_n(E)\, c_{q_1}(F/F_{q_1}) \\
&= \frac{1-\beta T}{2-\beta T}(1-\beta T)^{-n+1-q_1}c_{n+q_1}(V-D_1-F_{q_1}) - \frac{1}{2}c_{n+q_1}(V/F_{q_1})  \\
& \quad + \frac{1}{2}c_n(F^*)\, c_{q_1}(F/F_{q_1}) \\
&= \frac{1-\beta T}{2-\beta T}(1-\beta T)^{-n+1-q_1}c_{n+q_1}(V-D_1-F_{q_1}) ,
\end{align*}
using the relation
\[
  c_n(F^*) = (-1)^{\dim(E\cap F)}\gamma(E,F)\, c_n(E)
\]
in the penultimate line.  This establishes the claim.

Next, using the identity
\begin{align*}
  c(D_{i-1}^\perp/D_{i-1} - D_i/D_{i-1} - F_{q_i}/D_{i-1}) &= c(V-D_i-F_{q_i})\cdot c(D_{i-1}-D_{i-1}^*) \\
  &= c(i) \prod_{j=1}^{i-1} \frac{(1-t_j)(1-\beta t_j)}{1+t_j-\beta t_j},
\end{align*}
as in type C, together with the fact that $\dim(E/D_{i-1}\cap F/D_{i-1}) = \dim(E\cap F) -i+1$ and the identity
\begin{align*}
 \gamma(E/D_{i-1},F/D_{i-1}) &= \gamma(E,F)\cdot c(D_{i-1};\beta) \\
    &= \gamma(E,F) \prod_{j=1}^{i-1} (1-\beta t_j),
\end{align*}
the factor $\zeta_i$ can be written
\begin{align*}
 \prod_{j=1}^{i-1}(1-\beta t_j) \left( \frac{1-\beta T}{2-\beta T}(1-\beta T)^{-\lambda_i}\left[ c(i)\prod_{j=1}^{i-1}\frac{1-t_j}{1+t_j-\beta t_j} \right]_{\lambda_i}  + \frac{1}{2}(-1)^i e(i) \right).
\end{align*}
Combined with the calculation of $t_i^a\cdot \zeta_i$, this proves the formula in the dominant case.

\medskip
The main case, where $k_i=i$ for all $i$, is deduced from the dominant case exactly as in type C.  Passing from this case to the general case is similar, using the relation provided by Lemma~\ref{l.drel}.  
\end{proof}

\begin{lemma}\label{l.drel}
Let $c(1)=c(2)=c(V-E_p-F_q)$ and $e(1)=e(2)=e(E_p,F_q)$, so $d(1)=c(1)-e(1)$ and $d(2)=c(2)+e(2)$.  Then for any $m$,
\[
  \frac{1-\delta_1\delta_2 R_{12}}{1+\delta_1\delta_2(R_{12}-\beta T_1)}\frac{(1-\beta\tilde T_1)^{-m+1}}{2-\beta\tilde T_1} \frac{(1-\beta\tilde T_2)^{-m+1}}{2-\beta\tilde T_2}\cdot d_\lambda(1)\, d_\lambda(2) = 0,
\]
where $\lambda=p+q$.  The same holds with $\tilde{d}(1)=c(1)+e(1)$ and $\tilde{d}(2)=c(2)-e(2)$ in place of $d(1)$ and $d(2)$.
\end{lemma}

\begin{proof}
To pass between $d$ and $\tilde{d}$, multiply the formula by the unit $(2\delta_1-1)(2\delta_2-1)$, which has the effect of switching the signs on $e(1)$ and $e(2)$.  Geometrically, this toggles between the cases where $\dim(E\cap F)$ is even or odd.  Let us assume the former, so $(-1)^{\dim(E\cap F)}=1$.

In the case $\lambda=m=0$, so $p=q=0$, then writing $\gamma=\gamma(E,F)$, the relation becomes
\begin{align*}
 & \frac{1-\delta_1\delta_2 R_{12}}{1+\delta_1\delta_2(R_{12}-\beta T_1)}\frac{1-\beta\tilde T_1}{2-\beta\tilde T_1} \frac{1-\beta\tilde T_2}{2-\beta\tilde T_2}\cdot (c_0 - \gamma(E,F) )\, (c_0+\gamma(E,F)) \\
  &\quad = \frac{1}{2-\beta\tilde T_1} \frac{1-\beta\tilde T_2}{2-\beta\tilde T_2} (c_0 - \gamma) (c_0 + \gamma ) \\
  &\quad = \left( \frac{1}{2}c(\beta/2) - \frac{1}{2}\gamma \right) \left( 1 - \frac{1}{2}\gamma(\beta/2) + \frac{1}{2}\gamma \right) \\
  &\quad = 0\cdot 1 = 0,
\end{align*}
as claimed.  The argument for the case where $\lambda>0$ and general $m$ proceeds just as in type C.
\end{proof}

We obtain a formula for skew-symmetric morphisms similar to the one for symmetric morphisms.  
Consider a morphism $\phi\colon W\to W^*$ so that the dual $\phi^\vee \colon W^{**} = W \to W^*$ is equal to $-\phi$, and suppose $W$ has rank $n$.  Let $c=c(W^*-W)$, and write
\begin{align*}
  c'_m &= c_m - \frac{1}{2}\beta c_{m+1} - \frac{1}{4}\beta^2 c_{m+2} - \frac{1}{8}\beta^3 c_{m+3} - \cdots\\
     &= c_m - \sum_{k > 0}\frac{1}{2^{k}}\beta^k\, c_{m+k}
\end{align*}
for any $m\in \ZZ$.  Since $c(W^*-W;\beta)=c(W^*;\beta)^2$, its canonical square root is $c(W^*-W;\beta/2)=c(W^*;\beta)$.  This means $c'_0 = 2-c(W^*;\beta)$, and for $m<0$, $c'_{m}=-(\beta/2)^{-m}\,c(W^*;\beta)$.

\begin{corollary}\label{c.pfD}
The locus $\Omega_r$ where $\rk(\phi)\leq r$ has class
\begin{align*}
  [{\Omega_r}] 
        &= \frac{1}{2^{n-r}}\Pf\left(m_{ij} \right)_{1\leq i<j\leq n-r},
\end{align*}
where $\lambda_i = n-r-i$ for $1\leq i\leq n-r$;
\begin{align*} 
  m_{ij} &=  c'_{\lambda_i} \, c'_{\lambda_j} + \sum_{m>0} \beta^m\, c'_{\lambda_i+m}\,c'_{\lambda_j}  \\ 
  &  \qquad + \sum_{\ell>0} \sum_{m\geq 0}  (-1)^\ell \left( \binom{\ell+m-1}{m} + \binom{\ell+m}{m} \right) \beta^m \, c'_{\lambda_i+\ell+m}\, c'_{\lambda_j-\ell}  \nonumber
\end{align*}
for $j<n-r$; and
\begin{align*} 
  m_{i,n-r} &=  \big(2-c(W^*;\beta)\big)\sum_{m\geq 0}\left(\frac{\beta}{2}\right)^m c_{\lambda_i+m} + c(W^*;\beta)\,c'_{\lambda_i} \\
  & \quad -c(W^*;\beta) \sum_{\ell>0}\sum_{m\geq 0} \left(-\frac{1}{2}\right)^\ell \left( \binom{\ell+m-1}{m} + \binom{\ell+m}{m} \right) \beta^{\ell+m} \, c'_{\lambda_i+\ell+m} . \nonumber
\end{align*}
\end{corollary}

\noindent
The proof is very similar to that of Corollary~\ref{c.pf}.

\begin{proof}
Let $V=W\oplus W^*$, equipped with the canonical symmetric form, $\langle v\oplus f, \, w\oplus g \rangle = f(w) + g(v)$, so the morphism $\phi$ is skew-symmetric if and only if its graph $E = \Gamma_\phi \subseteq V$ is isotropic.  Letting $F=W\oplus 0 \subseteq V$, we see $\rk(\phi)\leq r$ is equivalent to $\dim(E\cap F) \geq n-r$, so the locus $\Omega_r$ is the degeneracy locus associated to $\triple =(n-r,0,0)$.  The partition is $\lambda = (n-r-1,\ldots,1,0)$, i.e., $\lambda_i = n-r-i$, for $1\leq i\leq n-r$.

It is easiest to deduce the corollary from the raising operator version \eqref{e.raisingD} of Theorem~\ref{t.mainD}.  This formula reduces to
\begin{align*}
  \left( \prod_{1\leq i<j\leq n-r}  \frac{1-\delta_i\delta_j R_{ij}}{1+\delta_i \delta_j (R_{ij}-\beta T_i)}\right) \left(\prod_{i=1}^{n-r} \frac{1-\beta\tilde T_i}{2-\beta\tilde T_i}  d_{\lambda_i}(i) \right),
\end{align*}
which is the Pfaffian of a matrix whose $(i,j)$ entry is
\begin{equation}\label{e.dcorij}
  \frac{1-\delta_i\delta_j R_{ij}}{1+\delta_i \delta_j (R_{ij}-\beta T_i)} \cdot \frac{1-\beta\tilde T_i}{2-\beta\tilde T_i} \cdot \frac{1-\beta\tilde T_j}{2-\beta\tilde T_j}\cdot d_{\lambda_i}(i) \,  d_{\lambda_j}(j) .
\end{equation}
The Chern classes are $c(V-E-F) = c(V/E - F) = c(W^*-W)$.  The Euler class is $e=(-1)^{n-r}\gamma(E,F) = (-1)^{n-r}c(W^*;\beta)$, and it only appears in $m_{ij}$ when $j=n-r$.  Using $c'_0=2-c(W^*;\beta)$, the factor $(c'_{0}+(-1)^{n-r}e_0)$ reduces to $2$.  
From the definition,
\[
\frac{1}{2}c'_m = \frac{1-\beta T}{2-\beta T} c_m,
\]
so for $j<n-r$, the entry \eqref{e.dcorij} is
\[
 \frac{1-R_{ij}}{1+R_{ij}-\beta T_i} \cdot \frac{1}{4}\, c'_{\lambda_i}(i) \,  c'_{\lambda_j}(j) ,
\]
which equals $(1/4)m_{ij}$.  For $j=n-r$, one obtains $(1/4)m_{i,n-r}$ using
\begin{align*}
  \frac{1-\beta \tilde T}{2-\beta \tilde T}\cdot d_0 &= \frac{1}{2}\big( c'_0 + (-1)^{n-r}e_0 \big) \\
  &= 1
\end{align*}
and
\[
 \sum_{m\geq 0} \beta^m\,c'_{\lambda_i+m} = \sum_{m\geq 0} \left(\frac{\beta}{2}\right)^m c_{\lambda_i+m}. \qedhere
\]
\end{proof}
%

\section{$\beta$-theta and $\beta$-eta polynomials}\label{s.theta}

The above Pfaffian formulas for vexillary degeneracy loci of types B, C, and D do not include all loci coming from isotropic Grassmannians.  In general, in addition to conditions on intersections between two isotropic subspaces, one must impose conditions of the form
\[
  \dim(E_p \cap F_{q}) \geq k,
\]
where $E_p$ is isotropic and $F_{q}$ is coisotropic---that is, $F_{q}\supseteq F_{q}^\perp$.  

Conditions involving coisotropic subspaces introduce another layer of complication.  
The main difficulty lies in deciding which conditions are allowed, and which are redundant.  In the isotropic Grassmannian case, where $p$ is fixed, the loci can be indexed by {\em $p$-strict partitions}, following \cite{bkt1,bkt2}.  The more general loci we consider here could be called ``theta-vexillary'', because their classes (in type C) are flagged theta-polynomials.  They are described by certain triples $\triple=(\bk,\bp,\bq)$, where $\bk$ and $\bp$ are as before, but $\bq$ is allowed to include negative entries---corresponding to the coisotropic bundles.  The requirements on $\triple$ and the definition of the loci $\Omega_\triple$ are the same as in \cite{af2}.  We will quickly review the setup below, referring to [{\it op.~cit.}] for further explanation.

The Chern class formulas we obtain for theta-vexillary loci naturally extend the cohomology formulas from \cite{af2}, which are recovered by taking $\beta=0$.  In types B and C, the ``main cases'' of the Grassmannian formulas (where $k_i=i$ and $p_i=p$ for all $i$) again recover the results of \cite{himn}.  The ``general case'' is new, as are the type D formulas.  

Once one knows the formulas, the proofs are straightforward combinations of the arguments from \cite{af2} and the methods used in proving Theorems~\ref{t.mainA}--\ref{t.mainD} above, so here I will only give a very quick indication of the main point.

\subsection{Type C}

First we will describe the algebra, and then the geometry.

Given a sequence $\rho=(\rho_1,\ldots,\rho_\ell)$ of nonnegative integers, let $\rho'=(\rho'_1,\ldots,\rho'_{\ell'})$ be the sequence defined by
\[
  \rho'_i = \#\{j \,|\, \rho_j\geq i\}.
\]
(This is borrowed from notation for conjugate partitions.)  We will assume there is a distinguished index $k$ such that $\rho$ satisfies
\[
\rho_j=j-1  \qquad \text{ for } j\leq k
\]
and
\[
 \rho_k \geq \rho_{k+1}\geq \cdots \geq \rho_\ell \geq 0,
\]
and then define a raising operator
\begin{align*}
  R^{(\rho)} &= \prod_{1\leq i\leq \rho_j < j\leq \ell}\frac{1-\beta T_i}{1+R_{ij}-\beta T_i} \cdot \prod_{1\leq i<j\leq \ell} (1-R_{ij}) \\
 &= \prod_{j=1}^{\ell} (1-\beta T_j)^{\rho'_j} \cdot \prod_{j=k+1}^{\ell} \frac{\prod_{i=1}^{k-1} (1-R_{ij})}{\prod_{i=1}^{\rho_j} (1+R_{ij}-\beta T_i)} \\
 & \qquad  \times \prod_{1\leq i<j\leq k} \frac{ 1-R_{ij} }{1+R_{ij}-\beta T_i} \cdot \prod_{k\leq i<j\leq \ell} (1-R_{ij}).
\end{align*}
(The second formulation appears longer, but it is useful in exhibiting the Pfaffian and determinantal factors, as the two products on the last line.)

A \define{$\rho$-strict partition} is a partition $\lambda = (\lambda_1,\ldots,\lambda_\ell)$ such that $\lambda +\rho = (\lambda_1+\rho_1,\ldots,\lambda_\ell+\rho_\ell)$ is also a partition.  The \define{$\beta$-theta} polynomial for a $\rho$-strict partition is defined as
\[
  \Theta^{(\rho)}_\lambda(c(1),\ldots,c(\ell);\beta) := R^{(\rho)}\cdot \prod_{i=1}^\ell (1-\beta T_i)^{-\lambda_i} c_{\lambda_i}(i),
\]
where $c(i)=1+c_1(i)+c_2(i)+\cdots$ are indexed variables on which the operator $T_i$ acts.

For the geometry, we need to say what triples are allowed.  As defined in \cite{af2}, a \define{triple} is $\triple=(\bk,\bp,\bq)$, where the sequences
\begin{align*}
 0 < &k_1 < \cdots < k_s, \\
  & p_1 \geq \cdots \geq p_s > 0, \; \text{ and }\\
  & q_1 \geq \cdots \geq q_s
\end{align*}
satisfy several conditions.  The $q_i$ may not be zero, so there is a distinguished index $a$ so that $q_a>0>q_{a+1}$.  If $p_s=1$, then all $q$'s must be positive.

For each $i\leq a$, let $m(i) = \min\{ m\,|\, q_i+(k_i-k_{i-1}-1)\geq q_m\}$.  Then the negative values
\[
 -q_i, -q_i-1, \ldots, -q_i-(k_i-k_{m(i)}-1)
\]
may not occur among the $q_j$, for $j>a$.

For each $i>a$, define $\rho_{k_i} := k_j$, where $j$ is the index so that $q_j>-q_i>q_{j+1}$.  This is filled out to a sequence $\rho=\rho(\triple)$ by setting $\rho_k=k-1$ for $k\leq k_a$, and $\rho_k=\rho_{k_i}$ for $k>k_a$ and $k_{i-1}<k\leq k_i$.

Finally, we define a partition $\lambda=\lambda(\triple)$ by
\[
  \lambda_{k_i} = \begin{cases} p_i+q_i-1 & \text{ if }i\leq a; \\ p_i+q_i+k_i-1-\rho_{k_i} & \text{ if } i>a, \end{cases}
\]
and filling in other parts minimally, subject to the requirement
\[
  \lambda_1>\cdots>\lambda_{k_a}>\lambda_{k_a+1}\geq \cdots \geq \lambda_{k_s}>0.
\]
The final condition on $\triple$ requires that $\lambda$ be a $\rho$-strict partition.

Now let $V$ be a vector bundle of rank $2n$, with a symplectic form.  Associated to a triple $\triple$, we have flags of subbundles
\[
  E_{p_1}\subseteq \cdots \subseteq E_{p_s} \qquad \text{ and } \qquad F_{q_1} \subseteq \cdots \subseteq F_{q_s},
\]
where $E_{p_i}$ is isotropic and has rank $n+1-p_i$ for all $i$; $F_{q_i}$ is isotropic when $q_i>0$, and has rank $n+1-q_i$ in this case; and $F_{q_i}$ is coisotropic when $q_i<0$, and has rank $n-q_i$ in this case.  The degeneracy locus is defined as usual:
\[
  \Omega_\triple = \{ x\in X\,|\, \dim(E_{p_i}\cap F_{q_i}) \geq k_i \text{ for }1\leq i\leq s\}.
\]

Here is the formula.
\renewcommand{\thethm}{5C}
\begin{thm}\label{t.thetaC}
Given a triple $\triple$, let $\rho=\rho(\triple)$ and let $\lambda=\lambda(\triple)$ be the corresponding $\rho$-strict partition.  In $CK^*(X)$, we have
\begin{align*}
   [\Omega_\triple] &= \Theta^{(\rho)}_{\lambda}(c(1),\ldots,c(k_s);\beta) ,
\end{align*}
where $c(k_i) = c(V-E_{p_i}-F_{q_i})$, and $c(k)=c(k_i)$ when $k_{i-1}<k\leq k_i$.
\end{thm}

The proof is a modification of the proof of Theorem~\ref{t.mainC} along the lines of the argument for \cite[Theorem~2]{af2}.  The key step is the dominant case, where one finds
\[
 [\Omega] = \prod_{j=1}^s (1-\beta T_j)^{-\lambda_j} c_{\lambda_j}( D_{\rho_j}^\perp/D_{j-1} - D_j/D_{j-1} - F_{q_j}/D_{j-1} ).
\]
Manipulating the Chern classes $c( D_{\rho_j}^\perp/D_{j-1} - D_j/D_{j-1} - F_{q_j}/D_{j-1} )$ as in the proof of Theorem~\ref{t.mainC} leads to the raising operator $R^{(\rho)}$.

\subsection{Type B}

Here $V$ has rank $2n+1$ and is equipped with a nondegenerate quadratic form.  The conditions on subbundles $E_{p_i}$ and $F_{q_i}$ are the same as in type C, as is the definition of a triple $\triple$, partition $\lambda=\lambda(\triple)$, and associated degeneracy locus $\Omega_\triple$.  The formula is similar to the type C formula, when $2$ is inverted.

\renewcommand{\thethm}{5B}
\begin{thm}\label{t.thetaB}
With notation as above, in $CK^*(X)[\frac{1}{2}]$ we have
\begin{align*}
  [\Omega_\triple] &= \left(\prod_{i=1}^{k_a}\frac{1-\beta T_i}{2-\beta T_i}\right)\cdot \Theta^{(\rho)}_{\lambda}(c(1),\ldots,c(k_s);\beta) ,
\end{align*}
where $c(k_i) = c(V-E_{p_i}-F_{q_i}-M)$ if $i\leq a$, $c(k_i)=c(V-E_{p_i}-F_{q_i})$ if $i>a$, and $c(k)=c(k_i)$ when $k_{i-1}<k\leq k_i$.
\end{thm}

(As noted in \S\ref{s.typeB}, the line bundle $M=\det(V)$ is included mainly to indicate its role in the geometry: its Chern class is $2$-torsion, so it vanishes in $CK^*(X)[\frac{1}{2}]$.)

\subsection{Type D}

As defined in \cite{af2}, a \define{triple} of type D is $\triple=(\bk,\bp,\bq)$, with the sequences
\begin{align*}
 0 < &k_1 < \cdots < k_s, \\
  & p_1 \geq \cdots \geq p_s \geq 0, \; \text{ and }\\
  & q_1 \geq \cdots \geq q_s
\end{align*}
satisfying conditions similar to the type C case.  The requirements can be quickly described as follows: the value $-1$ is prohibited for the $q$'s, and if one forms a new triple $\triple^+$ by replacing each $p_i$ by $p_i+1$ and each nonnegative $q_i$ by $q_i+1$, then $\triple^+$ is a type C triple.

As before, there is an index $a$ such that $q_a>-1>q_{a+1}$.  The sequence $\rho=\rho(\triple)$ is defined similarly: For $i>a$, $\rho_{k_i} = k_j$, where $j$ is the index such that $q_j\geq -q_i > q_{j+1}+1$.  For $k\leq k_a$, $\rho_k=k-1$, and the rest of the entries are filled out by setting $\rho_k=\rho_{k_i}$ for $k_{i-1}<k\leq k_i$.

The partition $\lambda=\lambda(\triple)$ is defined by
\[
  \lambda_{k_i} = \begin{cases} p_i+q_i & \text{ if }i\leq a; \\ p_i+q_i+k_i-\rho_{k_i} & \text{ if } i>a, \end{cases}
\]
and then filling in other parts minimally so that $\lambda_1>\cdots>\lambda_{k_a}\geq \lambda_{k_a+1} \geq \cdots \geq \lambda_{k_s} \geq 0$.

The formulas will be in terms of \define{$\beta$-eta polynomials}.  Given $\rho$, with a distinguished index $k$ so that $\rho_j=j-1$ for all $j\leq k$, we define the raising operator
\begin{align*}
  \tilde{R}^{(\rho)}  &= \prod_{j=1}^{\ell} (1-\beta\tilde T_j)^{\rho'_j} \cdot \prod_{j=k+1}^{\ell} \frac{\prod_{i=1}^{k-1} 1-\delta_i\delta_j R_{ij}}{\prod_{i=1}^{\rho_j} 1+\delta_i\delta_j (R_{ij}-\beta T_i)} \\
 & \qquad  \times \prod_{1\leq i<j\leq k} \frac{ 1-\delta_i \delta_j R_{ij} }{1+\delta_i \delta_j(R_{ij}-\beta T_i)} \cdot \prod_{k\leq i<j\leq \ell} (1-\delta_i\delta_j R_{ij}).
\end{align*}

Given $\rho$ with distinguished index $k$, and a $\rho$-strict partition $\lambda$, the $\beta$-eta polynomial is
\begin{align*}
  \Eta_\lambda^{(\rho)}(d(1),\ldots,d(\ell);\beta) := \left(\prod_{i=1}^{k} \frac{1-\beta\tilde T_i}{2-\beta\tilde T_i}\right) \cdot \tilde{R}^{(\rho)}\cdot \prod_{i=1}^{\ell} (1-\beta\tilde T_i)^{-\lambda_i}\,d_{\lambda_i}(i),
\end{align*}
where the variables $d(i)=c(i)\pm e(i)$ are as in \S\ref{s.typeD}, and for $i>k$, we set $e(i)=0$.

The vector bundle $V$ has rank $2n$ and a nondegenerate quadratic form.  The isotropic bundles $E_{p_i}$ and $F_{q_i}$ (with $p_i,q_i\geq 0$) are as in \S\ref{s.typeD}, and for $q_i<0$, the bundle $F_{q_i}$ is coisotropic of rank $n-q_i$.  The degeneracy locus is
\[
  \Omega_\triple = \{ x\in X\,|\, \dim(E_{p_i}\cap F_{q_i}) \geq k_i \text{ for } 1\leq i\leq s \},
\]
as usual, with the same caveat about interpreting this as in \S\ref{s.typeD}.

\renewcommand{\thethm}{5D}
\begin{thm}\label{t.etaD}
Given a type D triple $\triple$, with $\rho=\rho(\triple)$ and $\lambda=\lambda(\triple)$, we have
\[
  [\Omega_\triple] = \Eta^{(\rho)}_\lambda( d(1),\ldots,d(k_s);\beta)
\]

\medskip
\noindent
in $CK^*(X)[\frac{1}{2}]$.  Here $d(k)=c(k)+(-1)^k e(k)$ for $k\leq k_a$ and $d(k)=c(k)$ for $k>k_a$, where the Chern classes are $c(k_i)=c(V-E_{p_i}-F_{q_i})$ and $e(k_i)=e(E_{p_i},F_{q_i})$, setting $c(k)=c(k_i)$ and $e(k)=e(k_i)$ for $k_{i-1}<k\leq k_i$.
\end{thm}

\newpage

\appendix
\section{Connective $K$-theory}\label{s.kthy}

Connective $K$-theory for schemes was introduced by Cai \cite{cai}, and an alternative version adapted to motivic homotopy theory was studied by Dai and Levine \cite{dl}.  These authors construct functors which extend Quillen's higher algebraic $K$-theory.  A simpler development is possible if one is concerned only with Grothendieck and Chow groups, which suffices for many applications to degeneracy loci.  The direct approach has the added benefit of requiring neither quasi-projective hypotheses (imposed in \cite{cai}) nor an appeal to resolution of singularities and moving lemmas (used in \cite{dl}).  

This appendix begins with a summary of the basic properties, especially those needed for degeneracy loci.  Precise definitions and an outline of the construction follow.  This material seems hard to find in the literature, so I have included a fair amount of detail.

\subsection{Basic properties}

All schemes are assumed to be separated and of finite type over a field.  Flat morphisms are assumed to have constant relative dimension, as in \cite{f-it}.  We will follow the sign conventions of \cite{himn} (which differ from those of \cite{cai,dl}).

\medskip
\noindent
{\it Homology.}  
For a scheme $X$, $CK_*X$ is a graded module over $\ZZ[\beta]$, with $\beta$ in degree $1$, and with $CK_i X=0$ for $i<0$.  This is covariant for proper maps: if $f\colon X \to Y$ is proper, then there are homomorphisms $f_*\colon CK_iX \to CK_iY$, compatible with the $\ZZ[\beta]$-module structure.

For $X=\pt$, there is a natural identification $CK_*(\pt) = \ZZ[\beta]$, the standard module.

In general, there is a natural identification $CK_{\dim X} X = K_\circ X$, and multiplication by $\beta$ determines isomorphisms $CK_i X \isom CK_{i+1} X \isom K_\circ X$ for all $i\geq \dim X$.

There is a {\it fundamental class} $[X]\in CK_{\dim X} X$, defined to be the class corresponding to $[\OO_X]\in K_\circ X$ under the identification $CK_{\dim X}X = K_\circ X$.  More generally, if $V\subseteq X$ is a closed subscheme of dimension $d$, then there is a class $[V]\in CK_d X$ obtained by pushforward from $CK_d V$.  The classes $[V]$ of subvarieties generate $CK_*X$ as a $\ZZ[\beta]$-module.  If $f\colon X' \to X$ is a proper map, restricting to a birational map of subvarieties $V' \to V$, and if both $V'$ and $V$ have rational singularities, then $f_*[V']=[V]$ in $CK_d X$.

The functor $CK_*$ interpolates between Chow groups and Grothendieck groups, in that there are functorial isomorphisms
\[
  (CK_*X)/(\beta) \xrightarrow{\sim} A_*X \qquad \text{ and } \qquad (CK_*X)[\beta^{-1}] \xrightarrow{\sim} (K_\circ X)[\beta,\beta^{-1}].
\]
That is, setting $\beta=0$ recovers Chow, and setting $\beta=-1$ recovers $K$-theory.

\medskip
\noindent
{\it Cohomology.}  
The connective $K$-cohomology $CK^*X$ is a graded algebra over $\ZZ[\beta]$, with $\beta$ in degree $-1$, and with $CK^iX=0$ for $i>\dim X$.  This is contravariant for arbitrary maps.  It acts on homology, making $CK_*X$ a graded $CK^*$-module, via the cap product
\[
  CK^iX \otimes CK_j X \to CK_{j-i} X.
\]
The cap product is functorial via the projection formula: if $f\colon X \to Y$ is proper, then $f_*(f^*c \cap \alpha) = c\cap f_*\alpha$, for $c\in CK^*Y$ and $\alpha\in CK_*X$.

When $X$ is smooth, there are Poincar\'e isomorphisms
\[
  CK^i X \xrightarrow{\sim} CK_{\dim X-i} X
\]
given by $c\mapsto c\cap [X]$.

The functor $CK^*$ interpolates between (operational) Chow and $K$-cohomology: there are functorial isomorphisms
\[
  (CK^*X)/(\beta) \xrightarrow{\sim} A^*X \qquad \text{ and } \qquad (CK^*X)[\beta^{-1}] \xrightarrow{\sim} (\opK^\circ X)[\beta,\beta^{-1}],
\]
where $A^*X$ is the Fulton-MacPherson Chow ring, and $\opK^\circ X$ is the operational $K$-theory defined in \cite{ap}.  (When $X$ is smooth, one has $\opK^\circ X\isom K^\circ X$, and in this case, our version of $CK^*X$ agrees with the one defined in \cite{dl}.)

\medskip
\noindent
{\it Chern classes.}  
Connective $K$-theory is equipped with a theory of Chern classes.  The only properties we need are the following.

\begin{itemize}
\item For a vector bundle $E$ of rank $e$ on $X$, there are Chern classes $c_i(E) \in CK^i(X)$, with $c_0(E)=1$, and with $c_i(E)=0$ unless $0\leq i\leq e$.  

\medskip
\item Chern classes satisfy the Whitney formula: given a short exact sequence $0 \to E' \to E \to E'' \to 0$, one has $c(E) = c(E')\cdot c(E'')$.  Chern classes of virtual bundles are defined by extending the Whitney formula: $c(E-F)$ is defined as $c(E)/c(F)$.

\medskip
\item First Chern classes of line bundles satisfy
\[
  c_1(L\otimes M) = c_1(L) + c_1(M) + \beta\,c_1(L)\,c_1(M).
\]
Together with the fact that the first Chern class of the trivial line bundle is zero, means that
\[
  c_1(L^*) = \frac{-c_1(L)}{1+\beta\,c_1(L)}.
\]

\medskip
\item If $Z$ is the zero locus of a regular section of $E$, then $[Z]=c_e(E)\cap [X]$ in $CK_{\dim X-e}X$.
\end{itemize}

For a vector bundle $E$ of rank $e$, and a line bundle $L$, three main properties are particularly useful for degeneracy locus formulas:
\begin{enumerate}[(a)]
\item We have
\begin{align*}
  c_e(E\otimes L^*) &= \sum_{k\geq 0} \binom{-e}{k}(-\beta)^k c_{e+k}(E-L) \\
               &= (1-\beta T)^{-e} c_e(E-L),
\end{align*}
where $T$ is the operator which adds $1$ to the index of the Chern class---i.e., $T\cdot c_i(E-L) = c_{i+1}(E-L)$.  (The generalized binomial coefficient is defined by $(1+x)^m = \sum_{k\geq 0} \binom{m}{k} x^k$ for any $m\in \ZZ$.  The identity is easy to prove by using the splitting principle for $E$ together with the formula relating the first Chern class of a line bundle with that of its dual.  Alternatively, see \cite[Lemma~2.1]{himn}.)

\medskip
\item For any $a\geq 0$ and $b\geq e$,
\[
  (-c_1(L))^a\cdot c_b(E-L) = c_{a+b}(E-L).
\]
(This is a formal consequence of the Whitney sum formula.)

\medskip
\item For the projective bundle $\PP(E) \xrightarrow{\pi} X$, with $Q=\pi^*E/\OO(-1)$ the tautological quotient of rank $e-1$, we have
\[
  \pi_*\big( c(Q)\cap [\PP(E)] \big) = (1-\beta)^{e-1}[X],
\]
and in particular $\pi_*[\PP(E)] = (-\beta)^{e-1}[X]$.  (To prove this, it suffices to consider the specializations at $\beta=0$ and $\beta=-1$, i.e., Chow and $K$-theory.  The Chow case is classical.  For $K$-theory, one can express $c_i(Q)$ in terms of exterior powers of $Q^*$ as in Karoubi's text\footnote{Karoubi's convention for $c_i(V)$ corresponds to ours for $c_i(V^*)$.} \cite[IV, 2.18]{karoubi} and use the fact that $\pi_*(\exterior^kQ^*)=0$ for all $k$.) 
\end{enumerate}

When $X$ is smooth, one often tacitly invokes the Poincar\'e isomorphism identifying $CK^iX = CK_{\dim X-i} X$.  This allows one to speak of the fundamental class $[V]\in CK^d X$ for a subvariety $V$ of codimension $d$.

\subsection{Definitions and constructions}

Let $\cM_i(X)$ be the category of coherent sheaves on $X$ whose support has dimension at most $i$, so there are natural inclusions $\cM_i(X)\hookrightarrow \cM_{i+1}(X)$.  Following Cai, we define
\[
  CK_i(X) := \im\big( K_\circ(\cM_i(X)) \to K_\circ(\cM_{i+1}(X)) \big).
\]
Two observations are immediate from the definition.  First, $CK_i(X)=0$ if $i<0$.  Second, there are natural homomorphisms $CK_i(X) \to CK_{i+1}(X)$, which are isomorphisms whenever $i\geq \dim X$, and $CK_i(X)\isom K_\circ(X)$ in this case.  A third fact comes easily from d\'evissage: $CK_i(X)$ is generated by classes $[\OO_V]$, where $V\subseteq X$ is a subvariety of dimension at most $i$.

We will define the $\ZZ[\beta]$-algebra structure on $CK_*$ by letting $-\beta$ act as the homomorphism $CK_i \to CK_{i+1}$.  (Soon this will be given a geometric interpretation.  The sign of $\beta$ agrees with \cite{himn}, and is chosen to eliminate signs in other formulas.)  
The evident isomorphism of graded groups $CK_*(\pt)\isom\ZZ[\beta]$ is therefore compatible with the standard structure of $\ZZ[\beta]$ as a $\ZZ[\beta]$-module.

The key fact about $CK_*$ is that it is a {\it refined oriented Borel-Moore} functor, or {ROBM}, in the terminology of \cite{gk}.  This means it has certain pushforwards and pullbacks which satisfy natural compatibility axioms, and it guarantees that the associated operational cohomology theory has good properties.

\medskip
\noindent
{\it Pushforwards for proper maps.}  
If $f\colon X \to Y$ is proper, then $f_*\colon CK_iX \to CK_iY$ is defined by
\[
  f_*[\shfF] := \sum_{p\geq 0} (-1)^p [R^pf_*\shfF],
\]
for $\shfF$ in $\cM_iX$.  (Note that the support of $R^pf_*\shfF$ is contained in the image of the support of $\shfF$.)  These are functorial for the same reason as for the usual pushforwards in $K_\circ$.  Furthermore, $f_*$ is a homomorphism of graded $\ZZ[\beta]$-modules: multiplication by $-\beta$ simply means regarding $\shfF$ as determining a class in $K_\circ(\cM_{i+1}X)$, and similarly $R^pf_*\shfF$ determines a class in $K_\circ(\cM_{i+1}Y)$.

The isomorphism $CK_*(\pt)=\ZZ[\beta]$ identifies $-\beta$ with the class in $CK_1(\pt)$ represented by $\OO_{\pt}$, which is equal to $(-\beta)\cdot [\pt] = -\beta$.  This provides a geometric characterization of $\beta$: writing $\pi\colon \PP^1\to \pt$, we have $-\beta = \pi_*[\PP^1]$ (and more generally, $(-\beta)^n=\pi_*[\PP^n]$).

\medskip
\noindent
{\it Refined Gysin pullbacks.}  
The most delicate part of the theory is the construction of pullbacks.  One situation is easy: if $f\colon X \to Y$ is flat of relative dimension $d$, then there is a pullback $f^*\colon CK_iY \to CK_{i+d}X$, defined simply by $f^*[\shfF]=[f^*\shfF]$.

Before turning to the more general lci pullbacks, a few particular instances of flat pullback are worth noting.  First, fundamental classes of pure-dimensional schemes can be characterized by flat pullback: if $\pi\colon X \to \pt$ is the projection, then $\pi^*[\pt]=[X]$ in $CK_{\dim X}X$.  Next, the localization sequence for Grothendieck groups implies a corresponding sequence in $CK_*$: if $Z\subseteq X$ is closed, with open complement $U=X\setminus Z$, then the proper pushforward and flat pullback fit together in exact sequences
\[
  CK_i Z \to CK_i X \to CK_i U \to 0.
\]
Finally, there is a homotopy-invariance property: if $\pi\colon E \to X$ is a vector bundle of rank $e$, then $\pi^*\colon CK_iX \to CK_{i+e}E$ is an isomorphism for all $i$.

Now suppose we have a fiber square
\[
\begin{tikzcd}
  X' \ar{r}\ar{d} & Y' \ar{d}{g} \\
  X  \ar{r}{f} & Y
\end{tikzcd}
\]
with $f$ lci of relative codimension $d$.  Then there is a Gysin homomorphism $f^!\colon CK_i(Y') \to CK_{i-d}(X')$.

Since smooth pullback has already been defined, it will suffice to consider the case when $f$ is a regular embedding.  One may hope to define $f^![\shfF] = \sum_{p\geq 0}(-1)^p [\tor^Y_p(\OO_{X},\shfF)]$, just as the Gysin pullback is defined for $K_\circ$ (see \cite{ap})---but in general these Tor sheaves may have support of dimension larger than $i-d$.  To work around this, we use the Fulton-MacPherson deformation to the normal cone.  
We first need to define the pullback for divisors in two very special cases.  Let $\iota \colon D\hookrightarrow Y$ be the inclusion of a Cartier divisor, defined by a section of a line bundle $L$ on $Y$.

The first case is where the normal bundle $L|_D$ is trivial.  Let $V\subseteq Y$ be a subvariety with $\dim V \leq i$; it suffices to define $\iota^*\colon CK_i Y \to CK_{i-1}D$ for the classes $[\OO_V]$ determined by such $V$.  If $V\not\subseteq D$, define $\iota^*[\OO_V] = [\OO_{V\cap D}]=[\OO_V\otimes \OO_D]$.  (This sheaf has support of dimension at most $i-1$, and since $\tor^Y_1(\OO_V,\OO_D)=0$ for $V\not\subseteq D$, the definition agrees with the naive one.)  If $V\subseteq D$, then $\iota^*[\OO_V]=0$.  (This also agrees with the naive definition, since one can write $\OO_D=1-[L^*]$ in $K_\circ Y$, and then $\iota^*[\OO_V]=(1-[L^*])\cdot[\OO_V]=[\OO_V]-[L^*|_V]=0$.)  It follows that the composition $\iota^*\iota_*\colon CK_i D \to CK_{i-1} D$ is zero in this case.

The second case is where $Y$ is a line bundle with projection $\pi\colon Y \to D$, and $D$ is embedded by the zero section.  Here $\iota^*$ is defined to be $(\pi^*)^{-1}\colon CK_i Y \to CK_{i-1} D$, using the homotopy isomorphism.  

Returning to the general situation where $f\colon X \hookrightarrow Y$ is a regular embedding of codimension $d$, we will use the notation of \cite[\S\S5--6]{f-it}.  Let $C'=C_{X'/Y'}$ be the normal cone, and $M^\circ=M^\circ_{X'/Y'}$ the deformation space.  As remarked in \cite[\S5]{f-it}, the normal bundle to the divisor $C'$ in $M^\circ$ is trivial, so we have a pullback for this inclusion.  Deformation to the normal cone produces homomorphisms
\[
\begin{tikzcd}
 CK_{i+1} C' \ar{r}{\iota_*} & CK_{i+1} M^\circ  \ar{d}{\iota^*} \ar{r} & CK_{i+1}(Y' \times\AA^1) \\
  & CK_i C' & CK_i Y' \ar{u}{p^*} \ar[dashed]{l}{\sigma},
\end{tikzcd}
\]
where the top row is the localization exact sequence.  The flat pullback $p^*$ lifts to a well-defined specialization homomorphism $\sigma$ for the  reasons given in \cite[\S5]{f-it}; see \cite[\S18.2]{f-it}.

Finally, let $N=g^*N_{X/Y}$, so this is a vector bundle of rank $d$ on $X'$, and $j\colon C' \hookrightarrow N$ embeds as a closed subscheme.  The Gysin homomorphism $f^!$ is defined as the composition
\[
  CK_i Y' \xrightarrow{\sigma} CK_i C' \xrightarrow{j_*} CK_i N \xrightarrow{s^*} CK_{i-d}X',
\]
where $s$ is the zero section and $s^*=(\pi^*)^{-1}$ is the pullback defined via homotopy invariance.

\medskip
\noindent
{\it Specializing to $A_*X$ and $K_\circ X$.}  
As explained in \cite[Theorem~7.1]{cai}, the Brown-Gersten-Quillen spectral sequence induces functorial isomorphisms
\[
  (CK_iX)/(\beta CK_{i-1}X) \to A_iX .
\]
(See also \cite[\S V.9]{weibel}, where $X$ is only required to be a noetherian scheme.)  This isomorphism is compatible with refined Gysin homomorphisms.  Indeed, the above construction for $CK_*$ specializes to the one for $A_*$ given in \cite{f-it}.

On the other hand, inverting $\beta$ produces isomorphisms
\[
  (CK_*X)[\beta^{-1}] \to (K_\circ X)[\beta,\beta^{-1}],
\]
as also shown in \cite[Theorem~7.1]{cai}.  This is compatible with the refined Gysin homomorphisms for $K_\circ$ defined via alternating Tor sheaves, as in \cite{ap}.  In fact, the Gysin homomorphism $f^!\colon CK_i Y' \to CK_{i-d} X'$ can be computed as
\[
  f^![\shfF] = \sum_{p\geq 0} (-1)^p [\tor^Y_p(\OO_X,\shfF)]
\]
whenever $i-d\geq \dim X'$.  Indeed, the right-hand side defines a class in $CK_{i-d}X'$ in this case.  It agrees with the general construction because the Koszul resolution of $\OO_X$ as an $\OO_Y$-module, pulled back to $Y'$ and deformed to $N$, also computes the homomorphism $CK_i N \to CK_{i-d} X'$.

Since the refined Gysin homomorphisms for $CK_*$ agree with those for $A_*$ and $K_\circ$, they automatically satisfy various functoriality and commutativity properties proved in \cite[\S6]{f-it} and \cite[\S3]{ap}.

As a $\ZZ[\beta]$-module, $CK_*X$ is flat if and only if the natural surjection $A_*X\to Gr_*K_\circ X$ is an isomorphism.  (Here $Gr_*K_\circ X$ is the associated graded for the dimension filtration; see \cite[Ex.~15.1.5]{f-it}.)

\medskip
\noindent
{\it Operational cohomology.}  
To define $CK^*X$, we apply the bivariant construction introduced by Fulton and MacPherson to the homology theory $CK_*$.  Given a morphism $f\colon X \to Y$, a bivariant class $c\in CK^i(X\xrightarrow{f} Y)$ is a collection of homomorphisms $c_g\colon CK_j Y' \to CK_{j-i}X'$, one for each morphism $g\colon Y' \to Y$, with $X'=X\times_Y Y'$ the fiber product.  These are required to satisfy various compatibility axioms; a good general reference for the construction in this context is \cite{gk}.  The cohomology theory is defined to be $CK^*X = CK^*(X \xrightarrow{\id} X)$.  This is contravariant in $X$ (for arbitrary morphisms of schemes).

The properties of the refined Gysin pullbacks for $CK_*$ amount to the statement that this forms a ROBM homology theory.  As shown in \cite{gk}, it follows that there are Poincar\'e isomorphisms $CK^i X \xrightarrow{\sim} CK_{\dim X-i} X$ when $X$ is smooth.  In particular, $CK^*(\pt)$ is isomorphic to $\ZZ[\beta]$, but with the opposite grading, so $\beta$ has degree $-1$.

There are {\it Kimura exact sequences} for $CK^*$, as in \cite{ap}.  If $p\colon X' \to X$ is an {\it envelope}, meaning that every subvariety of $X$ is the birational image of a subvariety of $X'$, then the sequence
\[
  0 \to CK^*X \xrightarrow{p^*} CK^*(X') \xrightarrow{p_1^*-p_2^*} CK^*(X'\times_X X')
\]
is exact, where $p_1$ and $p_2$ are the two projections.  Similarly, if the fiber square
\[
\begin{tikzcd}
  E \ar{d} \ar[hook]{r} & X' \ar{d} \\
  S \ar[hook]{r} & X
\end{tikzcd}
\]
is an {\it abstract blowup square}, meaning that the induced map $X'\setminus E \to X\setminus S$ is an isomorphism, then the sequence
\[
  0 \to CK^*X \to CK^*S \oplus CK^*X' \to CK^*E
\]
is exact.

In a category that admits resolution of singularities, the Kimura sequences imply that $CK^*$ is the (right) Kan extension of the same functor defined on smooth schemes.  Concretely, this means that if $L^*$ is any cohomology theory which agrees with $CK^*$ on smooth schemes, then there are natural homomorphisms $L^*X \to CK^*X$.  In particular, there are natural homomorphisms
\[
  CK^*_{\mathrm{DL}}X \to CK^*X,
\]
where $CK^*_{\mathrm{DL}}$ is the connective $K$-cohomology constructed by Dai and Levine, and these are isomorphisms when $X$ is smooth.

\medskip
\noindent
{\it Chern classes.}  
If $L$ is a line bundle on $X$, with zero section $s\colon X \to L$, the homomorphism $s^*s_*\colon CK_i X \to CK_{i-1} X$ defines the first Chern class operator $c_1(L)\in CK^1X$.  Via the usual Grothendieck construction, this extends canonically to a theory of Chern classes for vector bundles.  The argument given in \cite[Theorem~6.4]{cai} shows that the group law for $CK^*$ is as claimed above: $c_1(L\otimes M) = c_1(L)+c_1(M)+\beta\, c_1(L)\, c_1(M)$.

More generally, for any rank $e$ vector bundle $E$ and any section $s$, one has
\[
c_e(E) = s^*s_*\colon CK_i X \to CK_{i-e} X.
\]
If the section is regular, with zero locus $Z$, then $[Z]=c_e(E)\cap [X]$.  (These facts are proved by the same arguments as for Chow groups, given in \S6 and \S14 of \cite{f-it}.)


\section{Euler classes and quadric bundles}\label{s.euler}

\setcounter{equation}{0}
\renewcommand{\theequation}{B.\arabic{equation}}
\setcounter{thm}{0}
\renewcommand{\thethm}{\arabic{thm}}

In this appendix we develop a $K$-theoretic extension of some results of Edidin and Graham about characteristic classes for quadratic forms \cite{eg}.  Specifically, we construct an {\em Euler class} for an even rank vector bundle $V$ with a nondegenerate quadratic form $q$, in the case $(V,q)$ is locally trivial in the Zariski topology.  

The key insight is that while the Chow Euler classes of \cite{eg} are defined up to sign, Euler classes in $CK^*$ are defined up to {\em unit} in $CK^0$.  Since classes in $CK^{>0}$ are nilpotent, the units in $CK^0$ are $\pm 1 + \beta CK^1$.  (The units in $A^0$ are just $\pm 1$, so this is a direct generalization from \cite{eg}.)  
On the other hand, we will see that it is necessary to invert $2$ to construct Euler classes in general---even when the pair $(V,q)$ is locally trivial in the Zariski topology.

Throughout this appendix, let $V$ be a vector bundle of rank $2n$ over a variety $X$, equipped with a nondegenerate quadratic form $q$, so that the pair $(V,q)$ is locally trivial in the Zariski topology.

\begin{def*}
An \define{Euler class} for $(V,q)$ is a class $\eu
$ in $CK^n(X)$ such that for any $f\colon X' \to X$ and any maximal isotropic subbundle $E\subseteq f^*V$, we have $f^*\eu= a\cdot c_n(E)$, for some unit $a\in CK^0(X')$.
\end{def*}

To construct the Euler class, we need a formula comparing top Chern classes of maximal isotropic subbundles.  This extends a conjecture of Fulton for Chow classes, which was proved in \cite{eg}.  Let us write $c(u)=c(V-E-F;u)$ for the total Chern class in $CK^*(X)$, where $E$ and $F$ are maximal isotropic.  We will use various specializations of $c(u)$ throughout this appendix.

\begin{thm}\label{t.sqrt}
Let $E$ and $F$ be two maximal isotropic subbundles of $V$.

\begin{enumerate}
\item The class $c(\beta)$ has a unique functorial square root $\gamma(E,F)$ in $CK^0(X)$ which evaluates to $1$ at $\beta=0$.  (Functoriality means that $f^*\gamma(E,F) = \gamma(f^*E,f^*F)$ for any $f\colon X' \to X$.)

\medskip

\item If $D$ is a vector bundle contained in $E\cap F$, then $\gamma(E/D,F/D)=\gamma(E,F)\cdot c(D;\beta)$ is the canonical square root of $c(D^\perp/D - E/D-F/D;\beta)$.

\medskip

\item We have $c_n(F^*) = (-1)^{\dim(E\cap F)}\gamma(E,F)\,c_n(E)$.
\end{enumerate}
\end{thm}

\begin{proof}
The class $\gamma(E,F)$ will be constructed below, using Theorem~\ref{t.push}.  We will prove the other properties, assuming $\gamma(E,F)$ exists.

First, functoriality guarantees that $\gamma(E,F)$ is unique if it exists.  Suppose $2$ is not a zerodivisor in $CK^*(X)$.  If $x,y\in CK^0(X)$ are such that $x^2=y^2$ and $x\equiv y\equiv 1$ modulo $\beta$, then $x+y\equiv 2$ modulo $\beta$, so $x+y$ is not a zerodivisor, and it follows that $x-y=0$.  In general, one can find a scheme $X'$ with morphisms $f\colon X' \to X$ and $g\colon X' \to Y$, so that $f^*$ is injective on $CK^*$, $2$ is not a zerodivisor in $CK^*(Y)$, and $f^*E=g^*E'$ and $f^*F=g^*F'$ for some bundles $E'$ and $F'$ on $Y$.  (See \cite[Lemma~1.6]{totaro}, or the final paragraph of \cite{graham}.)

To see (ii), use $D^\perp\isom (V/D)^*$ and $V\isom V^*$ to write $c(D^\perp/D - E/D - F/D;\beta) = c(V-E-F;\beta)\cdot c(2D;\beta) = c(V-E-F;\beta)\cdot c(D;\beta)^2$.

To prove (iii), one can assume that $F=L_1\oplus \cdots \oplus L_n$ splits as a direct sum of line bundles.  (Indeed, there is a morphism $f\colon X' \to X$ so that $f^*$ is injective on $CK^*$ and so that $f^*F$ splits; now replace $X$ by $X'$.)  Furthermore, the maximal isotropic subbundle $E\subseteq V$ determines a section of the Grassmann bundle $OG(n,V)$ parametrizing isotropic subbundles $E'\subseteq V$ such that $\dim(E'\cap F)\equiv\dim(E\cap F) \pmod 2$.  Replacing $X$ by $OG(n,V)$, we may assume $E$ is the tautological subbundle on $OG(n,V)$.

The claim now follows by an equivariant localization argument.  With $F$ splitting as above, let $z_i=c_1(L_i)$.  For each $I\subseteq \{1,\ldots,n\}$ of cardinality congruent to $\dim(E\cap F)$ mod $2$, there is a point $p_I$ of $OG(n,V)$, namely, the section corresponding to the isotropic subbundle $\bigoplus_{i\in I} L_i \oplus \bigoplus_{i\not\in I} L_i^*$.  Each restriction $F_I := E|_{p_I}$ satisfies (iii), because
\begin{align*}
  (-1)^{\dim(F_I\cap F)} \gamma(F_I,F)\, c_n(F_I) &= (-1)^{|I|} \sqrt{\frac{c(F^*;\beta)}{c(F_I;\beta)}} \left( \prod_{i\in I} z_i \right)\left( \prod_{i\not\in I} \frac{-z_i}{1+\beta z_i} \right) \\
    &=  (-1)^n\,c_n(F) \left( \prod_{i\not\in I} \frac{1}{1+\beta z_i} \right)   \\
    &\qquad \times \sqrt{\prod_{i=1}^n \frac{1}{1+\beta z_i}  \prod_{i\in I}\frac{1}{1+\beta z_i} \prod_{i\not\in I}(1+\beta z_i) } \\
    &=   (-1)^n\,c_n(F) \left( \prod_{i\not\in I} \frac{1}{1+\beta z_i} \right) \sqrt{\prod_{i\in I} \frac{1}{(1+\beta z_i)^2}  } \\
    &= c_n(F^*) .
\end{align*} 
It follows that $E$ satisfies (iii).
\end{proof}

It remains to construct the class $\gamma(E,F)$.  For this, we use a pushforward formula.

\begin{thm}\label{t.push}
Let $E$ and $F$ be maximal isotropic subbundles of $V$, and let $\pi\colon \cQ(V) \to X$ be the quadric bundle.  Let $e=[\PP(E)]$ and $f=[\PP(F)]$ be the classes in $CK^{n-1}(\cQ(V))$.  Then
\begin{align*}
 \pi_*(ef) &= \begin{cases} 1 & \text{ if } \dim(E\cap F) \text{ is odd; } \\ 1-\gamma(E,F) & \text{ if } \dim(E\cap F) \text{ is even,}\end{cases}
\end{align*}
where $\gamma(E,F)$ is the canonical square root of $c(\beta)$.

That is, when $\dim(E\cap F)$ is even, $1-\pi_*(ef)$ squares to $c(V-E-F;\beta)$ and reduces to $1$ mod $\beta$.
\end{thm}

This constructs $\gamma(E,F)$ when $E\cap F$ is even-dimensional.  For the case where $\dim(E\cap F)$ is odd, take a line bundle $L$ on $X$, endow $V\oplus L\oplus L^*$ with the standard quadratic form, and consider $E'=E\oplus L$ and $F'=F\oplus L$.  Then $\dim(E'\cap F')$ is even, and by (ii) of Theorem~\ref{t.sqrt}, we have $\gamma(E,F)=\gamma(E',F')\cdot(1+\beta c_1(L))$.

\begin{proof}[Proof of Theorem~\ref{t.push}]
This is another localization argument.  As in the proof of Theorem~\ref{t.sqrt}(iii), we may assume $F=L_1\oplus \cdots \oplus L_n$, and replace $X$ by the Grassmann bundle $OG(n,V)$ and $E$ by the tautological subbundle.  The pushforward $\pi_*\colon CK^*(\cQ(V)) \to CK^*(X)$ can be computed by an Atiyah-Bott-Berline-Vergne summation.

Fix a set $I\subseteq \{1,\ldots,n\}$ with $|I|$ having the same parity as $\dim(E\cap F)$.  As before, it suffices to show that the formula holds for each restriction $F_I = E|_{p_I}$.  That is, we compute $\pi_*(f_I\cdot f)$, where $f_I = [\PP(F_I)]$.

For each $i\in \{1,\ldots,n\}$, there are points $p_i$ and $p_{-i}$ of $\cQ$, namely, the sections corresponding to the isotropic line bundles $L_i$ and $L_i^*$, respectively.  One has $p_i\in \PP(F_I)$ iff $i\in I$, and $p_{-i}\in \PP(F_I)$ iff $i\not\in I$.  The localization formula computes $\pi_*(f_I\cdot f)$ as
\[
  \pi_*(f_I\cdot f) = \sum_{i=1}^n \left( \frac{f_I|_i\cdot f|_i}{c_{2n-2}(T_i\cQ)} + \frac{f_I|_{-i}\cdot f|_{-i}}{c_{2n-2}(T_{-i}\cQ)}\right).
\]
Since $p_{-i}\not\in \PP(F)$, we have $f|_{-i}=0$ for all $i$, and similarly $f_I|_i=0$ unless $i\in I$.  The sum can therefore be written as
\[
  \pi_*(f_I\cdot f) = \sum_{i\in I} \frac{f_I|_i\cdot f|_i}{c_{2n-2}(T_i\cQ)}.
\]

The tangent bundle to $\cQ$ at $p_i$ is
\[
  T_i\cQ = \Hom(L_i,V/(L_i\oplus L_i^*)),
\]
and for $i\in I$, the normal bundle to $\PP(F_I)$ at $p_i$ is
\[
  N_i = \Hom(L_i,V/(F_I\oplus L_i^*)),
\]
It follows that
\[
   \frac{f_I|_i\cdot f|_i}{c_{2n-2}(T_i\cQ)} = \prod_{j\in I, j\neq i}\frac{-z_i-z_j-\beta z_i z_j}{(1+\beta z_j)(z_j-z_i)}.
\]
Summing over $i$, one has the formal identity
\begin{align*}
 \sum_{i\in I} \prod_{j\in I, j\neq i}\frac{-z_i-z_j-\beta z_i z_j}{(1+\beta z_j)(z_j-z_i)} &= \begin{cases} 1 &\text{ if } |I| \text{ is odd;} \\ \displaystyle{1 -\prod_{i\in I} \frac{1}{1+\beta z_i}} &\text{ if } |I| \text{ is even.} \end{cases}
\end{align*}
Since $c(V-F_I-F) = \left(\prod_{i\in I} \frac{1}{1+\beta z_i}\right)^2$, the claim is proved.
\end{proof}

We will need a relation among Chern classes of $V$, in the presence of isotropic subbundles.  Let $E$ and $F$ be maximal isotropic, and $c(u)=c(V-E-F;u)$.

\begin{lem}
We have $c(u)\cdot c(\beta-u) = c(\beta)$.
\end{lem}

\begin{proof}
The quadratic form identifies $V/E \isom E^*$, so
\[
  c(u) = c(V-E-F;u) = \frac{c(E^*;u)}{c(F;u)} = \frac{c(F^*;u)}{c(E;u)}.
\]
Using the general formula
\[
  c(E^*;u) = \frac{c(E;\beta-u)}{c(E;\beta)},
\]
we can write
\begin{align*}
  c(u) \cdot c(\beta-u) &= \frac{c(E^*;u)}{c(F;u)} \cdot \frac{c(F^*;\beta-u)}{c(E;\beta-u)} \\
  &= \frac{c(E^*;u)}{c(F;u)}\cdot \frac{c(F;u)\,c(F^*;\beta)}{c(E^*;u)\,c(E;\beta)} \\
  &= \frac{c(F^*;\beta)}{c(E;\beta)} \\
  &= c(\beta),
\end{align*}
as claimed.
\end{proof}

\renewcommand{\theequation}{\fnsymbol{equation}}
Specializing $u=\beta/2$ gives the formula
\begin{equation}\label{e.half}
\gamma(E,F)=c(\beta/2)
\end{equation}
in $CK^*(X)[\frac{1}{2}]$.  Indeed, the lemma shows $c(\beta/2)^2=c(\beta)$, and since $c(\beta/2)$ evaluates to $1$ at $\beta=0$, it must equal $\gamma(E,F)$ by Theorem~\ref{t.sqrt}.  
(The proof of the Lemma shows that these relations hold when $V$ has a symplectic form, or indeed any nondegenerate bilinear form.  What is special about even orthogonal bundles is that the class $\gamma(E,F)$ exists {\em integrally}.)

Now we can prove the existence of Euler classes.

\begin{thm}\label{t.euler}
For any even-rank bundle with quadratic form $(V,q)$ which is trivial in the Zariski topology, there is an Euler class $\eu\in CK^n(X)[\frac{1}{2}]$.
\end{thm}

\begin{proof}
If $V$ has a maximal isotropic subbundle $E$, then $\eu=c_n(E)$ is an Euler class.  Indeed, this satisfies the defining condition: if $F$ is another maximal isotropic bundle, then $c_n(F) = \pm \gamma(E,F^*)\, c_n(E)$, so $c_n(E)$ and $c_n(F)$ agree up to units.

In general, let $p\colon OG(n,V) \to X$ be the Grassmann bundle parametrizing one family of maximal isotropic subspaces.  Since $V$ is locally trivial, this is an envelope, and there is a {\em Kimura exact sequence}
\[
  0 \to CK^*(X) \xrightarrow{p^*} CK^*(OG) \xrightarrow{p_1^*-p_2^*} CK^*(OG\times_X OG),
\]
where $p_1$ and $p_2$ are the two projections.    
Let $E\subseteq V$ be the tautological bundle on $OG$, and let $E_1=p_1^*E_1$ and $E_2=p_2^*E_2$ be the two tautological bundles on the product.  Since $c_n(E_1) = \gamma(E_1^*,E_2)\cdot c_n(E_2)$ by Theorem~\ref{t.sqrt}(iii), and
\begin{align*}
  \gamma(E_1^*,E_2) &= c(V-E_1^*-E_2;\beta/2) \\
    &= \frac{c(E_2^*;\beta/2)}{c(E_1^*;\beta/2)},
\end{align*}
the class $c(E^*;\beta/2)\cdot c_n(E)$ lies in the kernel of $p_1^*-p_2^*$, and is therefore the pullback of a class in $CK^*(X)[\frac{1}{2}]$.
\end{proof}

It is in fact necessary to invert $2$ to construct $\eu$.  For bundles $(V,q)$ which are not Zariski locally trivial, Totaro gave an example showing that there is no integral Euler class in the Chow ring \cite[Proof of Theorem 5.1]{to2}.  In $K$-theory, integrality already fails for bundles which are Zariski-locally trivial, and the reason is simpler.  The first nontrivial example shows why this is so.

\begin{ex*}[{cf.~\cite[Example, p.~293]{eg}}]
First, we observe that in the above construction, the choice of $c(E^*;\beta/2)$ as the coefficient of $c_n(E)$ on $OG$ is essentially unique, up to a unit which is pulled back from $X$.  Indeed, suppose $\alpha\in CK^0(OG)$ is a unit so that $\alpha\cdot c_n(E)$ is pulled back from $X$.  Writing $\alpha_i=p_i^*\alpha$ for $i=1,2$, after inverting $2$ we must have $\alpha_1/\alpha_2 = c(E_1^*;\beta/2)/c(E_2^*;\beta/2)$.  Equivalently, $\alpha_1/c(E_1^*;\beta/2)=\alpha_2/c(E_2^*;\beta/2)$, so by the Kimura sequence, the ratio $\alpha/c(E^*;\beta/2)$ is pulled back from $X$.

To show that inverting $2$ is necessary, it will suffice to produce an example where the Chern classes of the dual tautological bundle $E^*$ are not divisible by $2$, and cannot be made divisible after multiplying by a unit pulled back from $X$.  

Now let $X\subseteq \PP^5$ be a smooth quadric hypersurface, with tautological line bundle $U\subset \CC^6$.  We will show that the rank $4$ bundle $V=U^\perp/U$ has no integral Euler class.  As shown in \cite{eg}, $V=U^\perp/U$ has no isotropic subbundle of rank $2$.  This will also follow from our computation, since a bundle which admits a maximal isotropic subbundle does have an integral Euler class, the top Chern class of the maximal isotropic subbundle.

Choose one of the maximal orthogonal Grassmann bundles $p\colon OG(V) \to X$, and let $E\subseteq p^*V$ be the tautological rank $2$ bundle.  Observe that $X\isom Gr(2,\CC^4)$, and under this isomorphism $U\isom \exterior^2 S_2$ and $\CC^6 \isom \exterior^2\CC^4$, where $S_2\subset \CC^4$ is the tautological bundle.  The bundle $OG(V) \to X$ is isomorphic to $Fl(2,3;\CC^4) \to Gr(2,\CC^4)$, and $E$ is identified with $\exterior^2 S_3/\exterior^2 S_2$, where $S_3\supset S_2$ is the tautological rank $3$ bundle.

One has presentations
\begin{align*}
  CK^*(Gr(2,\CC^4)) &= \ZZ[\beta][c_1,c_2,d_1,d_2]/I \quad \text{ and } \\
  CK^*(Fl(2,3,\CC^4)) &= \ZZ[\beta][c_1,c_2,x_3,x_4]/J ,
\end{align*}
where $c_i=c_i(S_2^*)$, $d_i=c_i( (\CC^4/S_2)^*)$, $x_3=c_1((S_3/S_2)^*)$, and $x_4=c_1((\CC^4/S_3)^*)$, and the relations are
\begin{align*}
I &= \big((1+c_1+c_2)(1+d_1+d_2)-1 \big) \quad \text{ and } \\
J &= \big( (1+c_1+c_2)(1+x_3)(1+x_4) - 1 \big).
\end{align*}
(These are the same as the Chow or cohomology presentations, for the same reasons.  See \cite[\S IV.3]{karoubi}.)  Since $\beta$ does not appear in the relations, we can temporarily give these rings alternative gradings by setting its $\deg(\beta)=0$.  One sees that $c_1$ and $x_3$ form a $\ZZ[\beta]$-basis for $CK^1(Fl)$ (in this grading), and $c_2$, $c_1^2$, and $c_1 x_3$ form a $\ZZ[\beta]$-basis for $CK^2(Fl)$.  
Using the relations, one computes the Chern polynomial of $E^*$ as
\begin{align*}
  c(E^*;u) 
      &= 1 + (c_1+2x_3)u + (c_2 + c_1 x_3 + x_3^2)u^2 \\
      &= 1 + (c_1+2x_3)u + (2c_2-c_1^2)u^2.
\end{align*}

Returning to the usual grading on $CK^*$, let $\alpha = 1+a_1\beta + a_2\beta^2 + O(\beta^3)$ be a unit in $CK^0(Gr)[\frac{1}{2}]$.  Using the fact that $c_1$ and $x_3$ are an integral basis, if $p^*\alpha\cdot c(E^*;\beta/2)$ is to be integral, we must have $a_1=-c_1/2$.  But then
\[
  p^*\alpha\cdot c(E^*;\beta/2) = 1 + x_3\,\beta + \left(\frac{1}{2}c_2-\frac{1}{2}c_1^2-\frac{1}{2}c_1\,x_3 + p^*a_2 \right) \beta^2 + O(\beta^3),
\]
which cannot be integral for any choice of $a_2\in CK^2(Gr)$, since the classes $c_2$, $c_1^2$, and $c_1 x_3$ are a basis for $CK^2(Fl)$ and $p^*a_2$ lies in the span of $c_2$ and $c_1^2$.
\end{ex*}




\end{document}